\documentclass[a4paper,draft]{amsart}
\usepackage{amssymb}
\usepackage{mathrsfs}
\usepackage[matrix,arrow]{xy}

\numberwithin{equation}{section}

\newtheorem{thm}{Theorem}[section]
\newtheorem{prop}[thm]{Proposition}
\newtheorem{lem}[thm]{Lemma}
\newtheorem{cor}[thm]{Corollary}
\theoremstyle{definition}
\newtheorem{defn}[thm]{Definition}
\theoremstyle{remark}
\newtheorem{rem}[thm]{Remark}

\newcommand*{\Sheaf}{\mathcal{SH}}
\newcommand*{\ModCat}[1]{#1\textrm{\normalfont -mod}}
\newcommand*{\C}{\mathbb{C}}
\newcommand*{\Z}{\mathbb{Z}}
\DeclareMathOperator{\im}{Im}
\DeclareMathOperator{\Ker}{Ker}
\DeclareMathOperator{\supp}{supp}
\DeclareMathOperator{\End}{End}
\DeclareMathOperator{\Hom}{Hom}
\DeclareMathOperator{\id}{id}
\DeclareMathOperator{\Id}{Id}
\DeclareMathOperator{\Res}{Res}
\DeclareMathOperator{\rank}{rank}

\DeclareMathOperator{\Cok}{Cok}

\author{Noriyuki Abe}
\title{The category $\mathcal{O}$ for a general Coxeter system}
\address{Graduate School of Mathematical Sciences, the University of Tokyo, 3--8--1 Komaba, Meguro-ku, Tokyo 153--8914, Japan.}
\email{abenori@ms.u-tokyo.ac.jp}

\begin{document}
\maketitle
\begin{abstract}
We study the category $\mathcal{O}$ for a general Coxeter system using a formulation of Fiebig.
The translation functors, the Zuckerman functors and the twisting functors are defined.
We prove the fundamental properties of these functors, the duality of Zuckerman functor and generalization of Verma's result about homomorphisms between Verma modules.
\end{abstract}

\section{Introduction}
The Bernstein-Gelfand-Gelfand (BGG) category $\mathcal{O}$ is introduced in~\cite{MR0407097}.
Roughly speaking, it is a full-subcategory of the category of modules of a semisimple Lie algebra which is generated by the category of highest weight modules.
Soergel~\cite{MR1029692} realized the endomorphism ring of the minimal progenerator of a block of $\mathcal{O}$ as the endomorphism ring of some module over the coinvariant ring of the Weyl group.
As a corollary, a block of the category $\mathcal{O}$ depends only on the attached Coxeter system (the integral Weyl group) and the singularity of the infinitesimal character.

Generalizing this method, Fiebig~\cite{MR2370278} and Soergel~\cite{MR2329762} construct some module over some algebra for any Coxeter system $(W,S)$.
If we consider the case of a Weyl group, the endomorphism ring of this module is equal to that of the minimal progenerator of the deformed category $\mathcal{O}$.
Specializing it, we get the category $\mathcal{O}$.

In this paper, we study the category $\mathcal{O}$ for a general Coxeter system.
Let $(W,S)$ be a Coxeter system and take a reflection faithful representation $V$ of $(W,S)$ (see \ref{subsec:Moment graph associated to a Coxeter system}).
After Braden-MacPherson~\cite{MR1871967}, we consider the associated moment graph.
Let $Z$ be the space of global sections of the structure algebra of this moment graph and $\{B(x)\}_{x\in W}$ the space of global sections of Braden-MacPherson sheaves.
Then $Z$ is an $S(V^*)$-algebra and $B(x)$ is a $Z$-module.
Consider a $\C$-algebra $A = \End_Z(\bigoplus_{x\in W}B(x))\otimes_{S(V^*)}\C$.
If $(W,S)$ is the Weyl group of a semisimple Lie algebra, then the regular integral block of the BGG category is equivalent to the category of finitely generated right $A$-modules.
However, in general case, the author dose not know whether the algebra $A$ is Noetherian.
Instead of this, we define a category $\mathcal{O}$ as the category of right $A$-modules.
By the above reason, even if $(W,S)$ is the Weyl group of a semisimple Lie algebra, $\mathcal{O}$ is not equivalent to the ordinal BGG category.

We state our results.
Put $P(x) = \Hom_Z(\bigoplus_{y\in W}B(y),B(x))\otimes_{S(V^*)}\C$.
Then $P(x)$ is a projective object of $\mathcal{O}$ and it has the unique irreducible quotient $L(x)$.
In \cite{MR2395170}, the translation functor $\theta_s^Z$ of the category of $Z$-modules are defined for a simple reflection $s$.
Then the module $A' = \Hom_Z(\bigoplus_yB(y),\bigoplus_x\theta_s^Z B(x))\otimes_{S(V^*)}\C$ is an $A$-bimodule.
Define a functor $\theta_s$ from $\mathcal{O}$ to $\mathcal{O}$ by $\theta_s(M) = \Hom_A(A',M)$.
Then we have the following theorem.
\begin{thm}[Proposition~\ref{prop:property of translation functors in widetilde O}, Theorem~\ref{thm:translation of projective modules}]
Let $s$ be a simple reflection and $x\in W$.
\begin{enumerate}
\item The functor $\theta_s$ is self-adjoint and exact.
\item If $xs < x$, then $\theta_s(P(x)) = P(x)^{\oplus 2}$.
\item The module $\theta_s L(x)$ is zero if and only if $xs > x$.
\end{enumerate}
\end{thm}

Next, we consider the Zuckerman functor.
Fix a simple reflection $s$ and let $\mathcal{O}_s$ be a full-subcategory of $\mathcal{O}$ consisting of a module $M$ such that $\Hom_A(P(x),M) = 0$ for all $sx < x$.
Then it is easy to see that the inclusion functor $\iota_s\colon \mathcal{O}_s\to \mathcal{O}$ has the left adjoint functor $\widetilde{\tau}_s$.
Put $\tau_s = \iota_s\circ\widetilde{\tau_s}$ and let $L\tau_s$ be its left derived functor.
Let $D^b(\mathcal{O})$ be the bounded derived category of $\mathcal{O}$.
We prove the following duality theorem.
\begin{thm}[Theorem~\ref{thm:duality of Zuckerman functor}]\label{thm:main theorem, duality of Zuckerman functor}
\begin{enumerate}
\item For $i > 2$ and $M\in\mathcal{O}$, we have $L^i\tau_s(M) = 0$.
Hence $L\tau_s$ gives a functor from $D^b(\mathcal{O})$ to $D^b(\mathcal{O})$.
\item The functor $L\tau_s[-1]$ is self-adjoint.
\end{enumerate}
\end{thm}
In the case of $\mathfrak{g}$-modules, this theorem is proved by Enright and Wallach~\cite{MR563362} (in more general situation).

Next result is a generalization of Verma's result about homomorphisms between Verma modules~\cite{MR0218417}.
Let $V(x)$ be a Verma $Z$-module~\cite[4.5]{MR2370278}.
Put $M(x) = \Hom_Z(\bigoplus_{y\in W}B(y),V(x))\otimes_{S(V^*)}\C$.
Then $M(x)$ gives a generalization of the Verma module.
We prove the following theorem.
\begin{thm}[Theorem~\ref{thm:Homomorphisms between Verma modules}]\label{thm:Homomorphisms between Verma modules section 1}
We have
\[
	\Hom(M(x),M(y)) = 
	\begin{cases}
	\C & (y \le x),\\
	0 & (y\not\le x).
	\end{cases}
\]
Moreover, any nonzero homomorphism $M(x)\to M(y)$ is injective.
\end{thm}

Final results are about the twisting functors~\cite{MR1474841}.
For a simple reflection $s$, we will define a generalization of the twisting functor $T_s$ (Section~5).
We prove the following theorem.
\begin{thm}[Proposition~\ref{prop:LT_s RC_s and Ltau_s}, Theorem~\ref{thm:LT_s is an auto-equivalence}, Theorem~\ref{thm:T_s sonnano-kankei-nee}]\label{thm:properties of twisting functor}
Let $s$ be a simple reflection.
We denote the derived functor of $T_s$ by $LT_s$.
Let $D(\mathcal{O})$ be the derived category of $\mathcal{O}$.
\begin{enumerate}
\item $L^iT_s = 0$ for $i > 1$.
\item The functor $LT_s$ gives an auto-equivalence of $D(\mathcal{O})$.
\item For a reduced expression $w = s_1\dotsm s_l$, $T_{s_1}\dotsm T_{s_l}$ is independent of the choice of a reduced expression.
\end{enumerate}
\end{thm}
In the case of the original BGG category, this is proved in \cite{MR1474841,MR2032059}.

We summarize the contents of this paper.
We recall results of Fiebig~\cite{MR2395170,MR2370278} in Section \ref{sec:Preliminaries}.
The category $\mathcal{O}$ and the translation functors are defined in Section \ref{sec:The category O}, and the fundamental properties are proved.
We also define an another functor $\varphi_s$.
In Section \ref{sec:Zuckerman functor}, we prove Theorem~\ref{thm:main theorem, duality of Zuckerman functor}.
The definition of the twisting functors appears in Section~\ref{sec:THe functors T_s anc C_s}, and fundamental properties are proved.
Theorem~\ref{thm:Homomorphisms between Verma modules section 1} is proved in Section~\ref{sec:Homomorphisms between Verma modules}.
We prove Theorem~\ref{thm:properties of twisting functor} in Section~\ref{sec:More about the functors T_s C_s}.

\section{Preliminaries}\label{sec:Preliminaries}
In this section, we recall results of Fiebig~\cite{MR2395170,MR2370278}.

\subsection{Moment graphs and Sheaves}\label{subsec:Moment graphs and Sheaves}
Throughout this paper, we consider $S(V^*)$ as a graded algebra for a vector space $V$ with grading $\deg V^* = 2$.
We define the grading shifts $\langle k\rangle$ by $(M\langle k\rangle)_n = M_{n - k}$ where $M = \bigoplus_{n\in\Z}M_n$ is a graded module.

\begin{defn}
Let $V$ be a vector space.
A \emph{$V^*$-moment graph} $\mathcal{G} = (\mathcal{V},\mathcal{E},h_\mathcal{G},t_\mathcal{G},l_\mathcal{G})$ is given by
\begin{itemize}
\item an ordered set $\mathcal{V}$, called the set of vertices.
\item a set $\mathcal{E}$, called the set of edges.
\item a map $t_\mathcal{G},h_\mathcal{G}\colon \mathcal{E}\to \mathcal{V}$ such that $t_\mathcal{G}(E) > h_\mathcal{G}(E)$ for all $E\in \mathcal{E}$.
\item a map $l_\mathcal{G}\colon \mathcal{E}\to \mathbb{P}^1(V^*)$.
\end{itemize}
\end{defn}
For $E\in \mathcal{E}_\mathcal{G}$, we denote $l_\mathcal{G}(E)$ by $V^*_E$.

\begin{defn}
Let $V$ be a vector space and $\mathcal{G} = (\mathcal{V},\mathcal{E},h_\mathcal{G},t_\mathcal{G},l_\mathcal{G})$ a $V^*$-moment graph.
\begin{enumerate}
\item A \emph{sheaf} $\mathscr{M} = ((\mathscr{M}_x)_{x\in \mathcal{V}},(\mathscr{M}_E)_{E\in \mathcal{E}},(\rho^\mathscr{M}_{x,E}))$ on $\mathcal{G}$ is given by
\begin{itemize}
\item a graded $S(V^*)$-module $\mathscr{M}_x$.
\item a graded $S(V^*)/V^*_ES(V^*)$-module $\mathscr{M}_E$.
\item an $S(V^*)$-module homomorphism $\rho^\mathscr{M}_{x,E}\colon \mathscr{M}_x\to \mathscr{M}_E$ for $x\in \mathcal{V}$ and $E\in \mathcal{E}$ such that $x\in\{t_\mathcal{G}(E),h_\mathcal{G}(E)\}$.
\end{itemize}
\item Let $\mathscr{M},\mathscr{N}$ be sheaves on $\mathcal{G}$.
A \emph{morphism} $f = ((f_x)_{x\in \mathcal{V}},(f_E)_{E\in \mathcal{E}})\colon \mathscr{M}\to \mathscr{N}$ is given by
\begin{itemize}
\item an $S(V^*)$-homomorphism $f_x\colon \mathscr{M}_x\to \mathscr{N}_x$.
\item an $S(V^*)$-homomorphism $f_E\colon \mathscr{M}_E\to \mathscr{N}_E$.
\item $\rho^\mathscr{N}_{x,E}\circ f_x = f_E\circ \rho^\mathscr{M}_{x,E}$.
\end{itemize}
\end{enumerate}
\end{defn}

Define a sheaf $\mathscr{A}_\mathcal{G}$ on $\mathcal{G}$ by $\mathscr{A}_\mathcal{G} = ((S(V^*))_{x\in \mathcal{V}},(S(V^*)/V^*_ES(V^*))_{E\in\mathcal{E}},(\rho_{x,E}))$ where $\rho_{x,E}$ is the canonical projection.
This sheaf is called the \emph{structure sheaf}.

For a sheaf $\mathscr{M} = ((\mathscr{M}_x)_{x\in \mathcal{V}},(\mathscr{M}_E)_{E\in\mathcal{E}},(\rho^\mathscr{M}_{x,E}))$ on $\mathcal{G}$, we can attach the space of its \emph{global sections} by
\[
	\Gamma(\mathscr{M}) = 
	\left\{
	((m_x),(m_E))\in \prod_{x\in \mathcal{V}}\mathscr{M}_x\oplus\prod_{E\in \mathcal{E}}\mathscr{M}_\mathcal{E}\mid \rho^\mathscr{M}_{x,E}(m_x) = m_E
	\right\}
\]
Put $Z_\mathcal{G} = \Gamma(\mathscr{A}_\mathcal{G})$.
Then $Z_\mathcal{G}$ has the structure of a graded $S(V^*)$-algebra and $\Gamma$ defines a functor from the category of sheaves on $\mathcal{G}$ to $\ModCat{Z_\mathcal{G}}$, here $\ModCat{Z_\mathcal{G}}$ is the category of graded $Z_\mathcal{G}$-modules.
We also define the support of $\mathscr{M}$ by $\supp\mathscr{M} = \{x\in \mathcal{V}\mid \mathscr{M}_x\ne 0\}$.
The grading shifts for a sheaf is defined by $\mathscr{M}\langle k\rangle = ((\mathscr{M}_x\langle k\rangle)_{x\in \mathcal{V}},(\mathscr{M}_E\langle k\rangle)_{E\in\mathcal{E}},(\rho_{x,E}^{\mathscr{M}}))$.
Then we have $\Gamma(\mathscr{M}\langle k\rangle) = \Gamma(\mathscr{M})\langle k\rangle$.

Let $\mathcal{V}'$ be a subset of $\mathcal{V}$.
Put $\mathcal{E}' = \{E\in \mathcal{E}\mid h_\mathcal{G}(E)\in\mathcal{V}',\ t_\mathcal{G}(E)\in\mathcal{V}'\}$.
Then $\mathcal{G}' = (\mathcal{V}',\mathcal{E}',h_\mathcal{G}|_{\mathcal{E}'},t_\mathcal{G}|_{\mathcal{E}'},l_\mathcal{G}|_{\mathcal{E}'})$ is also a $V^*$-moment graph.
For a sheaf $\mathscr{M} = ((\mathscr{M}_x)_{x\in \mathcal{V}},(\mathscr{M}_E)_{E\in\mathcal{E}},(\rho^\mathscr{M}_{x,E}))$ on $\mathcal{G}$, $((\mathscr{M}_x)_{x\in \mathcal{V}'},(\mathscr{M}_E)_{E\in\mathcal{E}'},(\rho^\mathscr{M}_{x,E}))$ is a sheaf on $\mathcal{G}'$.
We denote this sheaf by $\mathscr{M}|_{\mathcal{V}'}$.

\subsection{$Z$-module with Verma flags}\label{subgsec:Z-module with Verma flags}
By the definition, we have $Z_\mathcal{G}\subset \prod_{x\in \mathcal{V}}S(V^*)$.
For $\Omega\subset \mathcal{V}$, let $Z_\mathcal{G}^\Omega$ be the image of $Z_\mathcal{G}$ under the map $\prod_{x\in \mathcal{V}}S(V^*)\to \prod_{x\in \Omega}S(V^*)$.
Let $\ModCat{Z_\mathcal{G}}^f$ be the category of graded $Z_\mathcal{G}$-modules that are finitely generated over $S(V^*)$, torsion free over $S(V^*)$ and the action of $Z_\mathcal{G}$ factors over $Z_\mathcal{G}^\Omega$ for a finite subset $\Omega\subset \mathcal{V}$.

Let $Q$ be the quotient field of $S(V^*)$.
Since $Z_\mathcal{G} \subset \prod_{x\in \mathcal{V}}S(V^*)$, we have $Z_\mathcal{G}\otimes_{S(V^*)}Q \subset \prod_{x\in \mathcal{V}}Q$.
We also have $Z_\mathcal{G}^\Omega\otimes_{S(V^*)}Q\subset \prod_{x\in\Omega}Q$.
\begin{lem}[{\cite[Lemma~3.1]{MR2370278}}]
If $\Omega$ is finite, then $Z_\mathcal{G}^\Omega\otimes_{S(V^*)}Q = \prod_{x\in\Omega}Q$.
\end{lem}
For $x\in \mathcal{V}$, put $e_x = (\delta_{xy})_y\in\prod_{y\in \mathcal{V}}Q$ where $\delta$ is Kronecker's delta.
Let $M$ be an object of $\ModCat{Z_\mathcal{G}}^f$ and take a finite subset $\Omega\subset \mathcal{V}$ such that the action of $Z_\mathcal{G}$ on $M$ factors over $Z_\mathcal{G}^\Omega$.
For $x\in\Omega$, put $M_Q^x = e_x(Q\otimes_{S(V^*)}M)$.
Set $M_Q^x = 0$ for $x\in\mathcal{V}\setminus \Omega$.
Then we have $M_Q = \bigoplus_{x\in \mathcal{V}}M_Q^x$ where $M_Q = Q\otimes_{S(V^*)}M$.
These are independent of a choice of $\Omega$.
Since $M$ is torsion-free, $M\subset M_Q$.

\begin{defn}
For $M\in\ModCat{Z_\mathcal{G}}^f$, $\Omega\subset\mathcal{V}$, put
\[
	M_\Omega = M\cap \bigoplus_{x\in \Omega}M_Q^x,
\]
and set
\[
	M^\Omega = \im\left(M\to M_Q\to \bigoplus_{x\in \Omega}M_Q^x\right).
\]
\end{defn}

A subset $\Omega\subset\mathcal{V}$ is called \emph{upwardly closed} if $x\in\Omega, y\ge x$ implies $y\in\Omega$.

\begin{defn}
We say that $M\in\ModCat{Z_\mathcal{G}}^f$ \emph{admits a Verma flag} if the module $M^\Omega$ is a graded free $S(V^*)$-module for each upwardly closed $\Omega$.
\end{defn}
Let $\mathcal{M}_\mathcal{G}$ be a full-subcategory of $\ModCat{Z_\mathcal{G}}^f$ consisting of the object which admits a Verma flag.

\begin{rem}
Fiebig~\cite{MR2395170,MR2370278} uses a notation $\mathcal{V}$ for the category of modules which admits a Verma flag.
Because we denote the set of vertices by $\mathcal{V}$, we use a different notation.
\end{rem}

The category $\mathcal{M}_\mathcal{G}$ is not an abelian category.
However, $\mathcal{M}_\mathcal{G}$ has a structure of an exact category~\cite[4.1]{MR2370278}.
\begin{defn}
Let $M_1\to M_2\to M_3$ be a sequence in $\mathcal{M}_\mathcal{G}$.
We say that it is short exact if and only if for each upwardly closed subset $\Omega$ the sequence $0\to M_1^\Omega\to M_2^\Omega\to M_3^\Omega\to 0$ is an exact sequence of $S(V^*)$-modules.
\end{defn}

\subsection{Localization functor}
Let $\Sheaf(\mathcal{G})$ be the category of sheaves $\mathscr{M}$ on $\mathcal{G}$ such that $\supp\mathscr{M}$ is finite and $\mathscr{M}_x$ is finitely generated and torsion free $S(V^*)$-module for each $x\in \mathcal{V}$.
Then we have $\Gamma(\Sheaf(\mathcal{G}))\subset \ModCat{Z}^f$.
\begin{prop}[Fiebig~\cite{MR2370278}]
The functor $\Gamma\colon \Sheaf(\mathcal{G})\to \ModCat{Z}^f$ has the left adjoint functor $\mathscr{L}$.
\end{prop}
The functor $\mathscr{L}$ is called the \emph{localization functor}.
%

For an image of $\mathcal{M}_\mathcal{G}$ under $\mathscr{L}$, we have the following proposition.
For a sheaf $\mathscr{M}$ on $\mathcal{G}$ and $x\in \mathcal{V}$, put
\[
	\mathscr{M}^{[x]} = \Ker\left(\mathscr{M}_x\to \bigoplus_{h_\mathcal{G}(E) = x}\mathscr{M}_E\right).
\]
A sheaf $\mathscr{M}$ is called \emph{flabby} if $\Gamma(\mathscr{M})\to \Gamma(\mathscr{M}|_\Omega)$ is surjective for all upwardly closed set $\Omega$.
\begin{prop}[\cite{MR2370278}]
\begin{enumerate}
\item The functor $\mathscr{L}$ is fully-faithful on $\mathcal{M}_\mathcal{G}$.
\item For $M\in \ModCat{Z_\mathcal{G}}^f$, put $\mathscr{M} = \mathscr{L}(M)$.
Then $M$ admits a Verma flag if and only if $\mathscr{M}$ is flabby and $\mathscr{M}^{[x]}$ is graded free for all $x\in\mathcal{V}$.
\end{enumerate}
\end{prop}

For $x\in\mathcal{V}$, define a sheaf $\mathscr{V}(x)$ by 
\begin{align*}
\mathscr{V}(x)_y & = 
\begin{cases}
S(V^*) & (y = x),\\
0 & (y\ne x),
\end{cases}\\
\mathscr{V}(x)_E & = 0.
\end{align*}
The sheaf $\mathscr{V}(x)$ is called a \emph{Verma sheaf} and its global section $V(x) = \Gamma(\mathscr{V}(x))$ is called a \emph{Verma module}.
The module $V(x)$ admits a Verma flag for all $x\in\mathcal{V}$.

\subsection{Projective object in $\mathcal{M}_\mathcal{G}$}\label{seubsec:Projective object in M}
Let $\mathcal{G} = (\mathcal{V},\mathcal{E},h_\mathcal{G},t_\mathcal{G},l_\mathcal{G})$ be a $V^*$-moment graph.
Since $\mathcal{M}_\mathcal{G}$ is an exact category, we can define the notion of a projective object in $\mathcal{M}_\mathcal{G}$.
We can also define the notion of a projective object in $\mathscr{L}(\mathcal{M}_\mathcal{G})$ since $\mathscr{L}$ is fully-faithful on $\mathcal{M}_\mathcal{G}$.

\begin{thm}[{\cite[Theorem~5.2]{MR2370278}}]\label{thm:classification of projective object}
For each $x\in \mathcal{V}$ there exists an indecomposable projective object $\widetilde{\mathscr{B}}(X)\in \mathscr{L}(\mathcal{M}_\mathcal{G})$ such that $\widetilde{\mathscr{B}}(x)_x\simeq S(V^*)$ and $\supp\widetilde{\mathscr{B}}(x)\subset\{y\mid y\le x\}$.

Moreover, a projective object in $\mathscr{L}(\mathcal{M}_\mathcal{G})$ is a direct sum of $\{\widetilde{\mathscr{B}}(x)\langle k\rangle\mid x\in\mathcal{V},\ k\in\Z\}$.
\end{thm}
The sheaf $\widetilde{\mathscr{B}}(x)$ is called a \emph{Braden-MacPherson sheaf}~\cite{MR1871967}.

\subsection{Moment graph associated to a Coxeter system}\label{subsec:Moment graph associated to a Coxeter system}
Let $(W,S)$ be a Coxeter system such that $S$ is finite.
We denote the set of reflections by $T$.
A finite dimensional representation $V$ of $W$ is called a \emph{reflection faithful representation} if for each $w\in W$, $V^w$ is a hyperplane in $V$ if and only if $w\in T$.
By Soergel~\cite{MR2329762}, there exists a reflection faithful representation.
Let $V$ be a reflection faithful representation.
For each $t\in T$, let $\alpha_t\in V^*$ be a non-trivial linear form vanishing on the hyperplane $V^t$.
If $s\ne t$, then $\alpha_s\ne \alpha_t$~\cite[Lemma~2.2]{MR2370278}.

Let $S'$ be a subset of $S$ and $W'$ the subgroup of $W$ generated by $S'$.
We attach a $V^*$-moment graph $\mathcal{G} = (\mathcal{V},\mathcal{E},h_\mathcal{G},t_\mathcal{G},l_\mathcal{G})$ to $((W,S),(W',S'))$ by 
\begin{itemize}
\item $\mathcal{V} = W/W'$, an order is induced by the Bruhat order.
\item $\mathcal{E} = \{\{xW',yW'\}\mid x\in TyW'\}$.
\item If $x\in Ty,x < y$, then $h_\mathcal{G}(\{xW',yW'\}) = xW',\ t_\mathcal{G}(\{xW',yW'\}) = yW'$.
\item $V^*_{\{xW',txW'\}} = \C\alpha_t$ for $xW'\in W/W'$, $t\in T$.
\end{itemize}

In the rest of this paper, we fix a Coxeter system $(W,S)$ and a reflection faithful representation $V$.
Let $\mathcal{G}$ be the $V^*$-moment graph associated to $((W,S),(\{e\},\emptyset))$.
Put $\mathscr{A} = \mathscr{A}_\mathcal{G}$, $Z = Z_\mathcal{G}$ and $\mathcal{M} = \mathcal{M}_\mathcal{G}$.

\subsection{Translation functor}\label{subsec:Translation functor}
We define an action of a simple reflection $s\in S$ on $\prod_{w\in W}S(V^*)$ by $s((z_w)_w) = (z_{ws})_w$.
This action preserves $Z$.
Put $Z^s = \{z\in Z\mid  s(z) = z\}$.
Then $Z^s$ is an $S(V^*)$-subalgebra.
For $M\in\ModCat{Z}^f$, put $\theta_s^ZM = Z\otimes_{Z^s}M\langle -1\rangle$.
Let $\widetilde{\mathscr{B}}(x)$ be the Braden-MacPherson sheaf and put $\mathscr{B}(x) = \widetilde{\mathscr{B}}(x)\langle -\ell(x)\rangle$
Set $B(x) = \Gamma(\mathscr{B}(x))$.

\begin{prop}[{\cite[Proposition~5.5, Corollary~5.7]{MR2395170}}]\label{prop:Fiebig's result about translation functor}
\begin{enumerate}
\item The functor $\theta_s^Z$ preserves $\mathcal{M}$.
\item The functor $\theta_s^Z$ is exact and self-adjoint.
\item For $M\in \ModCat{Z}^f$, $\supp(\mathscr{L}(\theta_s^Z(M)))\subset \supp(\mathscr{L}(M))\cup \supp(\mathscr{L}(M))s$.
\item Assume that $xs > x$.
There exists a projective object $P\in\mathcal{M}$ such that $\theta_s^Z(B(x)) = B(xs)\oplus P$ and $\supp\mathscr{L}(P)\subset \{y\in W\mid y\le x\}$.
\item There exist degree zero canonical homomorphism $\Id\langle 1\rangle\to \theta_s^Z$ and $\theta_s^Z\to \Id\langle -1\rangle$.
\end{enumerate}
\end{prop}
\begin{rem}\label{rem:explicit discription of the natural transformation of theta_s}
Set $c_s = (w(\alpha))_w$.
The natural transformation $\Id\langle 1\rangle\to \theta_s^Z$ is given by $m\mapsto c_s\otimes m + 1\otimes c_sm$ and $\theta_s^Z\to \Id\langle -1\rangle$ is given by $z\otimes m\mapsto zm$.
\end{rem}

\section{The category $\mathcal{O}$}\label{sec:The category O}

\subsection{The functor $\varphi_s^Z$}
For a graded $S(V^*)$-module $M$ and $w\in W$, let $b_w(M)$ be an $S(V^*)$-module whose structure map is given by $S(V^*)\xrightarrow{w} S(V^*)\to \End(M)$.
We remark that if $M$ is annihilated by $\alpha_t$ for $t\in T$, then we have $b_t(M) \simeq M$ as a graded $S(V^*)$-module.

First we define a functor $a_S\colon \Sheaf(\mathcal{G})\to \Sheaf(\mathcal{G})$ by the following.
Let $\mathscr{M}\in \Sheaf(\mathcal{G})$.
Then the sheaf $a_S(\mathscr{M})$ is defined by
\begin{itemize}
\item $(a_S(\mathscr{M}))_x = b_{x^{-1}}M_{x^{-1}}$ for $x\in W$,
\item $(a_S(\mathscr{M}))_E = b_{x^{-1}}(M_{E'})$ where $x = h_\mathcal{G}(E)$, $h_\mathcal{G}(E') = h_\mathcal{G}(E)^{-1} = x^{-1}$ and $t_\mathcal{G}(E') = t_\mathcal{G}(E)^{-1} = (tx)^{-1}$,
\item $\rho_{x,E}^{a_S(\mathscr{M})} = \rho_{x^{-1},E'}^{\mathscr{M}}$.
\end{itemize}
It is easy to see that these data define a sheaf $a_S(\mathscr{M})$ and functor $a_S\colon \Sheaf(\mathcal{G})\to \Sheaf(\mathcal{G})$.

Let $a_Z\colon \prod_{x\in W}S(V^*)\to \prod_{x\in W}S(V^*)$ be an algebra homomorphism defined by $a((z_w)_w) = (wz_{w^{-1}})_w$.
Then $a_Z$ preserves a subalgebra $Z$ and gives a $\C$-algebra homomorphism.
We remark that $a_Z$ is not an $S(V^*)$-algebra homomorphism.
For a $Z$-module $M$, let $a_M(M)$ be a $Z$-module whose structure map is given by $Z \xrightarrow{a} Z\to \End(M)$.
This defines a functor $a_M\colon \ModCat{Z}\to \ModCat{Z}$.
\begin{lem}\label{lem:relation between L,Gamma,a}
\begin{enumerate}
\item We have $\supp(a_S(\mathscr{M})) = \{x^{-1}\mid x\in \supp\mathscr{M}\}$.
\item We have $a_S(\Sheaf(\mathcal{G})^f)\subset \Sheaf(\mathcal{G})^f$.
\item We have $a_M(\ModCat{Z}^f)\subset \ModCat{Z}^f$.
\item We have $\Gamma\circ a_S \simeq a_M\circ \Gamma$.
\item We have $\mathscr{L}\circ a_M \simeq a_S\circ \mathscr{L}$.
\end{enumerate}
\end{lem}
\begin{proof}
(1) and (2) is obvious from the definition.

(3)
By the definition, we have $a_Z(Z^\Omega) = Z^{\Omega'}$ where $\Omega' = \{x^{-1}\mid x\in \Omega\}$.
Hence if the action of $Z$ on $M$ factors over $Z^\Omega$, the action on $a_M(M)$ factors over $Z^{\Omega'}$.

(4)
Let $\mathscr{M}\in \Sheaf(\mathcal{G})$.
By the definition, we have
\begin{multline*}
\Gamma(a_S(\mathscr{M})) \\= \left\{((m_x),(m_E))\in\prod_{x\in W}b_{x^{-1}}\mathscr{M}_{x^{-1}}\oplus\prod_{E\in\mathcal{E}}b_{x^{-1}}\mathscr{M}_{E'}\mid \rho_{x^{-1},E'}^\mathscr{M}(m_x) = m_E\right\},
\end{multline*}
where $E'$ is the same as in the definition of $a_S$.
Replace $x\mapsto x^{-1}$.
Then $E'$ becomes $E$.
Hence we get
\[
	\Gamma(a_S(\mathscr{M})) = \left\{((m_{x^{-1}}),(m_{E'}))\in\prod_{x\in W}b_{x}\mathscr{M}_x\oplus\prod_{E\in\mathcal{E}}b_{x}\mathscr{M}_{E}\mid \rho_{x,E}^\mathscr{M}(m_x) = m_E\right\}.
\]
From this formula, as a space, $\Gamma(a_S(\mathscr{M})) = \Gamma(\mathscr{M})$.
The action of $z = (z_w)\in Z$ on $((m_x),(m_E))\in \Gamma(a_S(\mathscr{M}))$ is given by $((x(z_{x^{-1}})m_x),(x(z_{x^{-1}})m_E))$ where $t_\mathcal{G}(E) = x$.
This action coincide with the action of $z$ on $a_M(\Gamma(\mathscr{M}))$.

(5)
Obviously, $a_S^2 = \Id$ and $a_M^2 = \Id$.
In particular, $a_S\colon \Sheaf(\mathcal{G})^f\to \Sheaf(\mathcal{G})^f$ and $a_M\colon \ModCat{Z}^f\to \ModCat{Z}^f$ are self-adjoint.
Hence, taking the left adjoint functor of the both sides in (4), we get (5).
\end{proof}

\begin{prop}\label{prop:a preserves M}
We have $a_M(\mathcal{M}) = \mathcal{M}$.
\end{prop}
\begin{proof}
Take $M\in\mathcal{M}$ and put $\mathscr{M} = \mathscr{L}(M)$, $\mathscr{N} = \mathscr{L}(a_M(M)) = a_S(\mathscr{M})$.
We prove that $\mathscr{N}$ is flabby and $\mathscr{N}^{[x]}$ is graded free for all $x\in W$.

Let $\Omega$ be a upwardly closed subset and put $\Omega' = \{x^{-1}\mid x\in\Omega\}$.
Then $\Omega'$ is also upwardly closed.
Since $\mathscr{M}$ is flabby, $\Gamma(\mathscr{M})\to \Gamma(\mathscr{M}|_{\Omega'})$ is surjective.
Hence $\Gamma(\mathscr{N}) = a_M(\Gamma(\mathscr{M}))\to a_M(\Gamma(\mathscr{M}|_{\Omega'})) = \Gamma(\mathscr{N}|_\Omega)$ is surjective.

By the definition of $\mathscr{N}^{[x]}$, we have $\mathscr{N}^{[x]} = b_{x^{-1}}(\mathscr{M}^{[x^{-1}]})$.
Since $\mathscr{M}^{[x^{-1}]}$ is graded free, $\mathscr{N}^{[x]}$ is graded free.
\end{proof}

\begin{lem}\label{lem:image of projective object by a}
We have $a_M(B(x)) = B(x^{-1})$.
\end{lem}
\begin{proof}
Since $a$ gives an auto-equivalence of the category $\mathcal{M}$, $a_M(B(x))$ is an indecomposable projective object.
By Lemma~\ref{lem:relation between L,Gamma,a} and the definition of $a_S$, we have $\supp\mathscr{L}(a_M(B(x))) = \supp a_S(\mathscr{L}(B(x))) = \{y^{-1}\mid y\in\supp\mathscr{L}(B(x))\}$ and $\mathscr{L}(a_M(B(x)))_{x^{-1}} = (a_S(\mathscr{L}(B(x))))_{x^{-1}} = b_{x^{-1}}\mathscr{L}(B(x))_x = b_{x^{-1}}S(V^*)\langle -\ell(x)\rangle = b_{x^{-1}}S(V^*)\langle -\ell(x^{-1})\rangle \simeq S(V^*)\langle -\ell(x^{-1})\rangle$.
Hence we get the lemma.
\end{proof}

From Proposition~\ref{prop:a preserves M}, we can define the functor $\varphi_s^Z\colon \mathcal{M}\to \mathcal{M}$ by $\varphi_s^Z = a_M\circ \theta_s^Z\circ a_M$.
Since $a_M$ gives an equivalence of categories, the fundamental properties of $\varphi_s^Z$ follows from that of $\theta_s^Z$.
\begin{prop}
\begin{enumerate}
\item The functor $\varphi_s^Z$ preserves $\mathcal{M}$.
\item The functor $\varphi_s^Z$ is exact and self-adjoint.
\item For $M\in\ModCat{Z}^f$, $\supp\mathscr{L}(\varphi_s^Z(M))\subset \supp\mathscr{L}(M)\cup s(\supp\mathscr{L}(M))$.
\item Assume that $sx > x$.
There exists a projective object $P\in\mathcal{M}$ such that $\varphi_s^Z(B(x)) = B(sx)\oplus P$ and $\supp\mathscr{L}(P)\subset\{y\in W\mid y\le x\}$.
\item There exist degree zero canonical homomorphisms $\Id\langle 1\rangle\to \varphi_s^Z$ and $\varphi_s^Z\to \Id\langle -1\rangle$.
\end{enumerate}
\end{prop}

We describe the functor $\varphi_s^Z$ more explicitly.
We define an algebra homomorphism $r_s\colon \prod_{w\in W}S(V^*)\to \prod_{w\in W}S(V^*)$ by $r_s((z_w)_w) = (s(z_{sw}))_w$.
Note that this is not an $S(V^*)$-module homomorphism.
The subalgebra $Z$ satisfies $r_s(Z) = Z$.
Recall that the map $s\colon Z\to Z$ is defined by $s((z_w)_w) = (z_{ws})_w$.
Then it is easy to see that $r_s\circ a_Z = a_Z \circ s$.
Set $Z^{r_s} = \{z\in Z\mid r_s(z) = z\}$.
Then we have $\varphi_s^Z M = Z\otimes_{Z^{r_s}}M$.
From this description, we get the following proposition.
\begin{prop}\label{prop:theta and varphi commute}
For simple reflections $s,t$, the functors $\theta_t^Z$ and $\varphi_s^Z$ commute with each other.
Moreover, the natural transformation $\theta_t^Z\langle 1\rangle \to \varphi_s^Z\theta_t^Z$ (resp.~$\varphi_s^Z\langle 1\rangle \to \theta_t^Z\varphi_s^Z$, $\varphi_s^Z\theta_t^Z\to \theta_t^Z\langle -1\rangle$, $\theta_t^Z\varphi_s^Z\to \varphi_s^Z\langle -1\rangle$) can be identified with $\theta_t^Z(\Id\langle 1\rangle\to \varphi_s^Z)$ (resp.~$\varphi_s^Z (\Id\langle 1\rangle\to\theta_t^Z)$, $\theta_t^Z(\varphi_s^Z\to \Id\langle -1\rangle)$, $\varphi_s^Z(\theta_t^Z\to\Id\langle -1\rangle)$).
\end{prop}
\begin{proof}
First we remark that $t$ and $r_s$ commute with each other.
Put $Z^{r_s,t} = Z^{r_s}\cap Z^t$.
We prove that $Z\otimes_{Z^{r_s,t}}M\simeq Z\otimes_{Z^{r_s}}Z\otimes_{Z^t}M$ for a $Z$-module $M$.
The same argument implies $Z\otimes_{Z^{r_s,t}}M\simeq Z\otimes_{Z^t}Z\otimes_{Z^{r_s}}M$.

Consider the map $\Xi\colon Z\otimes_{Z^{r_s,t}}M\to Z\otimes_{Z^{r_s}}Z\otimes_{Z^t}M$ defined by $\Xi(z\otimes m) = z\otimes 1\otimes m$.
This map is a $Z$-module homomorphism.
Set $\alpha = \alpha_s$.
We regard $\alpha$ as an element of $Z$ by the structure map $S(V^*)\to Z$.
Put $c_t = (w(\alpha_t))_w$.
Then we have $Z = Z^t\oplus c_tZ^t$~\cite[Lemma~5.1]{MR2395170}.
Since $a_Z(c_s) = \alpha_s$, we have $Z = Z^{r_s}\oplus \alpha Z^{r_s}$.
Hence we get
\[
	Z\otimes_{Z^{r_s}}Z\otimes_{Z^t}M = (1\otimes1\otimes M)\oplus(\alpha\otimes 1\otimes M)\oplus(1\otimes c_t\otimes M)\oplus(\alpha\otimes c_t\otimes M).
\]
Similarly, we get
\[
	Z\otimes_{Z^{r_s,t}}M = (1\otimes M)\oplus(\alpha\otimes M)\oplus(c_t\otimes M)\oplus(\alpha c_t\otimes M).
\]
Since $c_t\in Z^{r_s}$, $1\otimes c_t\otimes M = c_t\otimes 1\otimes M$ and $\alpha\otimes c_t\otimes M = \alpha c_t\otimes 1\otimes M$.
Hence $\Xi$ is an isomorphism.

We prove the second claim.
We omit a grading.
The map $Z\otimes_{Z^t}M\to Z\otimes_{Z^{r_s}}Z\otimes_{Z^t}M$ is given by $1\otimes m\mapsto 1\otimes \alpha \otimes m + \alpha\otimes1 \otimes m$ (Remark~\ref{rem:explicit discription of the natural transformation of theta_s}).
Since $\alpha\in Z^t$, we have $1\otimes \alpha\otimes m = 1\otimes 1\otimes \alpha m$.
Under the isomorphism $Z\otimes_{Z^t}Z\otimes_{Z^{r_s}}M\simeq Z\otimes_{Z^{r_s,t}}M\simeq Z\otimes_{Z^{r_s}}Z\otimes_{Z^t}M$, $z\otimes 1\otimes m\in Z\otimes_{Z^t}Z\otimes_{Z^{r_s}}M$ corresponds to $z\otimes 1\otimes m\in Z\otimes_{Z^{r_s}}Z\otimes_{Z^t}M$.
Hence the map $Z\otimes_{Z^t}M\to Z\otimes_{Z^{r_s}}Z\otimes_{Z^t}M \simeq Z\otimes_{Z^t}Z\otimes_{Z^{r_s}}M$ is given by $1\otimes m\mapsto 1\otimes 1\otimes \alpha m + \alpha\otimes 1\otimes m = 1\otimes 1\otimes \alpha m + 1\otimes \alpha\otimes m$.
This is equal to $\theta_t^Z(\Id\to \varphi_s^Z)$.
We can prove the other formulae by the same argument.
\end{proof}

\begin{lem}\label{lem:translation out of projectives in Z-mod}
Fix $s\in S$ and put $S' = \{s\}$, $W' = \{1,s\}$.
Let $\mathcal{G}'$ be the moment graph associated to $((W,S),(W',S'))$, $\widetilde{\mathscr{B}}'(xW')$ the Braden-MacPherson sheaf and $B'(xW') = \Gamma(\widetilde{\mathscr{B}}'(xW'))\langle -\ell(x)\rangle$ for $x\in W$ such that $xs < x$.
Using $Z_{\mathcal{G}'}\simeq Z^s$~\cite[5.1]{MR2395170}, we regard $B'(xW')$ as a $Z^s$-module.
If $xs < x$, $Z\otimes_{Z^s}B'(xW') \simeq B(x)$.
\end{lem}
\begin{proof}
Notice that $Z\otimes_{Z^s}\cdot $ and $\Res_{Z^s}$ have the exact right adjoint functors.
Hence they preserve a projective object.
By~\cite[Lemma~5.4]{MR2395170}, $\mathscr{L}(Z\otimes_{Z^s}B'(xW'))_x = S(V^*)\langle -\ell(x)\rangle$ and its support is contained in $\{y\in W\mid y \le x\}$.
Hence $B(x)$ is a direct summand of $Z\otimes_{Z^s}B'(xW')$.
Take a projective object $P$ such that $Z\otimes_{Z^s}B'(xW') = B(x)\oplus P$.
We prove $P = 0$.
In the rest of this proof, we omit a grading.
By the construction of the Braden-MacPherson sheaf~\cite[1.4]{MR1871967}, $\mathscr{L}(B(x))_{x} = \mathscr{L}(B(x))_{xs} = S(V^*)$.
By~\cite[Lemma~5.4]{MR2395170}, $\mathscr{L}(Z\otimes_{Z^s}B'(xW'))_x = \mathscr{L}(Z\otimes_{Z^s}B'(xW'))_{xs} = S(V^*)$.
Hence $\mathscr{L}(P)_{x} = \mathscr{L}(P)_{xs} = 0$.
Since $Z \simeq (Z^s)^{\oplus 2}$ as a $Z^s$-module~\cite[Lemma~5.1]{MR2395170}, we have $\Res_{Z^s}(Z\otimes_{Z^s}B'(xW')) = B'(xW')^{\oplus 2}$.
Therefore, if $P\ne 0$, then $\Res_{Z^s}(B(x)) = B'(xW')$ and $\Res_{Z^s}(P) = B'(xW')$.
Since $\mathscr{L}(P)_{x} = \mathscr{L}(P)_{xs} = 0$, we have $\mathscr{L}(\Res_{Z^s}(P))_{xW'} = 0$~\cite[Proposition~5.3]{MR2395170}.
This is a contradiction.
Hence $P = 0$.
\end{proof}

\begin{prop}\label{prop:translation of projective Z-module}
Let $s$ be a simple reflection and $x\in W$.
\begin{enumerate}
\item If $xs > x$, then $\theta_s^Z B(x) = B(xs)\oplus \bigoplus_{y< x,\ ys > y,\ k\in\Z}B(y)\langle k\rangle^{m_{y,k}}$ for some $m_{y,k}\in\Z_{\ge 0}$.
\item If $xs < x$, then $\theta_s^Z B(x) = B(x)\langle 1\rangle\oplus B(x)\langle -1\rangle$.
\item If $sx > x$, then $\varphi_s^Z B(x) = B(xs)\oplus \bigoplus_{y< x,\ sy > y,\ k\in\Z}B(y)\langle k\rangle^{m_{y,k}}$ for some $m_{y,k}\in\Z_{\ge 0}$.
\item If $sx < x$, then $\varphi_s^Z B(x) = B(x)\langle 1\rangle\oplus B(x)\langle -1\rangle$.
\end{enumerate}
\end{prop}
\begin{proof}
Let $W',S',B'(xW')$ be as in the previous lemma.

(1)
Since $\Res_{Z^s}B(x)$ is a projective object and the support of $\mathscr{L}(\Res_{Z^s}(B(x)))$ is contained in $\{yW'\mid y\le x\}$, we have $\Res_{Z^s}B(x) = \bigoplus_{k\in\Z}B'(xsW')\langle k\rangle^{m_k}\oplus \bigoplus_{y< x,\ ys > y,\ k\in\Z}B'(yW')\langle k\rangle^{m_{y,k}}$ for some $m_k$ and $m_{y,k}$.
Then by the previous lemma, we get $\theta_s^Z B(x) = \bigoplus_{k\in\Z}B(xs)\langle k-1\rangle^{m_k}\oplus \bigoplus_{y< x,\ ys > y,\ k\in\Z}B(y)\langle k-1\rangle^{m_{y,k}}$.
By Proposition~\ref{prop:Fiebig's result about translation functor}, we have $m_k = 0$ if $k\ne 1$ and $m_1 = 1$.

(2)
From \cite[Lemma~5.1]{MR2395170}, we have $\Res_{Z^s}(Z\otimes_{Z^s}\cdot) = \Id\oplus\Id\langle 2\rangle$.
Hence we have
\begin{multline*}
	\theta_s^Z B(x) = \theta_s^Z(Z\otimes_{Z^s}B'(xW')) = Z\otimes_{Z^s}(\Res_{Z^s}(Z\otimes_{Z^s}B'(xW')))\langle -1\rangle\\
	\simeq Z\otimes_{Z^s}(B'(xW')\langle 1\rangle\oplus B'(xW')\langle -1\rangle) \simeq B(x)\langle 1\rangle\oplus B(x)\langle -1\rangle.
\end{multline*}

(3) and (4) follows from (1) and (2) and Lemma~\ref{lem:image of projective object by a}.
\end{proof}

\subsection{Definition of the category $\mathcal{O}$}
Set $\widetilde{A} = \End_Z(\bigoplus_{x\in W}B(x))$.
This is an $S(V^*)$-algebra.
\begin{defn}
Put $A = \widetilde{A}\otimes_{S(V^*)}\C$ where $\C = S(V^*)/V^*S(V^*)$ is a one-dimensional $S(V^*)$-algebra.
Define the category $\mathcal{O}$ as the category of right $A$-modules.
\end{defn}
\begin{rem}\label{rem:relation between O and BGG category}
Even if $(W,S)$ is the Weyl group of some Kac-Moody Lie algebra, the category $\mathcal{O}$ is not equivalent to the Bernstein-Gelfand-Gelfand (BGG) category since BGG category has some finiteness conditions.
If $(W,S)$ is a finite Weyl group, then the category of finitely generated right $A$-modules is equivalent to the regular integral block of the BGG category.
More generally, if $(W,S)$ is the Weyl group of some Kac-Moody Lie algebra, a block of the BGG category with positive level can be recovered from the algebra $A$~\cite{MR2395170}.
\end{rem}
Let $\widetilde{\mathcal{O}}$ be the category of right $\widetilde{A}$-modules.
Since $A = \widetilde{A}/V^*\widetilde{A}$ is a quotient of $\widetilde{A}$, we regard $\mathcal{O}$ as a full-subcategory of $\widetilde{\mathcal{O}}$.

Define the functor $\widetilde{\Phi}\colon\ModCat{Z}\to \widetilde{\mathcal{O}}$ by $\widetilde{\Phi}(M) = \Hom_Z(\bigoplus_{x\in W}B(x),M)$ and put $\Phi(M) = \widetilde{\Phi}(M)\otimes_{S(V^*)}\C$.
\begin{lem}\label{lem:relation of Z-mod and O}
Let $P$ be a direct sum of $\{B(x)\mid x\in W\}$'s and $M\in\mathcal{M}$.
Then the following canonical maps are isomorphisms:
\begin{itemize}
\item $\Hom_Z(P,M)\to \Hom_{\widetilde{A}}(\widetilde{\Phi}(P),\widetilde{\Phi}(M))$.
\item $\Hom_Z(P,M)\otimes_{S(V^*)}\C\to \Hom_A(\Phi(P),\Phi(M))$.
\end{itemize}
\end{lem}
\begin{proof}
We may assume that $P = B(x)$ for some $x\in W$.
Hence it is sufficient to prove when $P = \bigoplus_{x\in W}B(x)$.
The lemma is obvious in this case.
\end{proof}
Set $\widetilde{P}(x) = \widetilde{\Phi}(B(x))$, $P(x) = \Phi(B(x)) = \widetilde{P}(x)\otimes_{S(V^*)}\C$, $\widetilde{M}(x) = \widetilde{\Phi}(V(x))$ and $M(x) = \Phi(V(x)) = \widetilde{M}(x)\otimes_{S(V^*)}\C$.
The module $M(x)$ is called \emph{a Verma module}.
The module $P(x)$ has the unique irreducible quotient.
The irreducible quotient is denoted by $L(x)$.
This is a one-dimensional $A$-module and the unique irreducible quotient of $M(x)$.
To summarize it, we get the following lemma.
\begin{lem}
\begin{enumerate}
\item $\widetilde{P}(x)$ is a projective $\widetilde{A}$-module.
\item $P(x)$ is a projective $A$-module.
\item $L(x)$ is a simple $A$-module (hence, simple $\widetilde{A}$-module).
\item We have $\Hom_A(P(x),L(y)) = \Hom_{\widetilde{A}}(\widetilde{P}(x),L(y)) = \delta_{xy}$.
\end{enumerate}
\end{lem}
\begin{proof}
For (4), notice that we have $\Hom_A(\widetilde{M}\otimes_{S(V^*)}\C,N) = \Hom_{\widetilde{A}}(\widetilde{M},N)$ for $\widetilde{M}\in\widetilde{\mathcal{O}}$ and $N\in\mathcal{O}$.
Hence we get $\Hom_A(P(x),L(y)) = \Hom_{\widetilde{A}}(\widetilde{P}(x),L(y))$.
\end{proof}

Since there exists a surjective morphism $B(x)\to V(x)$, we have a surjective map $P(x)\to M(x)$.
Moreover, we get the following proposition.

\begin{prop}\label{prop:projective module has Verma filtration}
For $x\in W$, there exists a submodules $0 = M_0\subset M_1\subset \dots \subset M_n = P(x)$ such that $M_i/M_{i - 1}\simeq M(x_i)$ for some $x_i\in W$.
Moreover, we can take $\{M_i\}$ such that $x = x_n\ge x_{n - 1}\ge\dotsb \ge x_1$.
\end{prop}
\begin{proof}
Consider the order filtration~\cite[4.3]{MR2370278} $\{N_i\}$ of $P(x)$.
Then we have $N_{i(v)}/N_{i(v) - 1}\simeq P(x)^{[v]}$.
Since $P(x)^{[v]} = V(v)^{n_v}$ for some $n_v\in \Z_{\ge 0}$, we get the proposition.
\end{proof}

\subsection{Translation functors}
In this subsection, we construct functors $\widetilde{\theta_s},\widetilde{\varphi_s}\colon\widetilde{\mathcal{O}}\to \widetilde{\mathcal{O}}$ using functors $\theta_s^Z,\varphi_s^Z$.
Since the construction is the same, set $F^Z = \theta_s^Z$ or $\varphi_s^Z$ and we will construct a functor $\widetilde{F}\colon \widetilde{\mathcal{O}}\to \widetilde{\mathcal{O}}$.

Put $\widetilde{A}' = \widetilde{\Phi}(\bigoplus_{y\in W}F^ZB(y))$.
Then the module $\widetilde{A}'$ is a right $\widetilde{A}$-module and left $\End(\bigoplus_{x\in W}F^ZB(x))$-module.
Moreover, using a homomorphism $\End(B(x))\to \End(F^ZB(x))$, $\widetilde{A}'$ is an $\widetilde{A}$-bimodule.
Define $\widetilde{F}\colon \widetilde{\mathcal{O}}\to \widetilde{\mathcal{O}}$ by $\widetilde{F}(\widetilde{M}) = \Hom_{\widetilde{A}}(\widetilde{A}',\widetilde{M})$ for $\widetilde{M}\in \widetilde{\mathcal{O}}$.
Then $\widetilde{F}(\widetilde{M})$ is a right $\widetilde{A}$-module.
Since $F^ZB(y)$ is a direct summand of $(\bigoplus_{x\in W}B(x))^{\oplus m}$ for some $m$, $\widetilde{A}'$ is a direct summand of $\widetilde{A}^{\oplus m}$ for some $m$.
Hence $\widetilde{A}'$ is a projective right $\widetilde{A}$-module.
This implies that $\widetilde{F}$ is an exact functor.

Set $B = \bigoplus_{y\in W}B(y)$.
From Lemma~\ref{lem:relation of Z-mod and O}, we have 
\begin{multline*}
\widetilde{A}'\simeq \Hom_{\widetilde{A}}(\widetilde{A},\widetilde{A}') = \Hom_{\widetilde{A}}(\widetilde{\Phi}(B),\widetilde{\Phi}(F^Z(B))) \\\simeq \Hom_Z(B,F^Z(B)) \simeq \Hom_Z(F^Z(B),B) \simeq \Hom_{\widetilde{A}}(\widetilde{A}',\widetilde{A}).
\end{multline*}
So we have $\widetilde{A}'\simeq \widetilde{F}(\widetilde{A})$.

Recall the following well-known lemma.
For the sake of completeness, we give a proof.
\begin{lem}\label{lem:realization of a right exact functor as a tensor product}
Let $R_1,R_2$ be an arbitrary ring, $\mathcal{C}_i$ the category of right $R_i$-modules ($i = 1,2$) and $G$ a right exact functor $\mathcal{C}_1\to\mathcal{C}_2$.
Then we have a functorial isomorphism $G(X)\simeq X\otimes_{R_1} G(R_1)$.
\end{lem}
\begin{proof}
From an $R_1$-module homomorphism
\[
	X \simeq \Hom_{R_1}(R_1,X)\to \Hom_{R_2}(G(R_1),G(X)),
\]
we have an $R_2$-module homomorphism $X\otimes_{R_1}G(R_1)\to G(X)$.
If $X$ is free, this map is an isomorphism.
For a general $X$, take an exact sequence $F_1\to F_0\to X\to 0$ such that $F_0,F_1$ are free.
Then we have the following diagram:
\[
\xymatrix{
F_1\otimes_{R_1}G(R_1)\ar[r]\ar[d] & F_0\otimes_{R_1}G(R_1)\ar[r]\ar[d] & X\otimes_{R_1}G(R_1)\ar[r]\ar[d] & 0\\
G(F_1)\ar[r] & G(F_0)\ar[r] & G(X)\ar[r] & 0.
}
\]
The left two homomorphisms are isomorphisms.
Hence $X\otimes_{R_1}G(R_1)\to G(X)$ is an isomorphism.
\end{proof}
Hence we have $\widetilde{F}(\widetilde{M})\simeq \widetilde{M}\otimes_{\widetilde{A}}\widetilde{F}(\widetilde{A}) \simeq \widetilde{M}\otimes_{\widetilde{A}}\widetilde{A}'$.
This implies
\[
	\Hom(\widetilde{F}\widetilde{M},\widetilde{N})\simeq \Hom(\widetilde{M}\otimes_{\widetilde{A}}\widetilde{A}',\widetilde{N})\simeq \Hom(\widetilde{M},\Hom_{\widetilde{A}}(\widetilde{A}',\widetilde{N})) = \Hom(\widetilde{M},\widetilde{F}\widetilde{N}).
\]
We get the following proposition.
\begin{prop}\label{prop:property of translation functors in widetilde O}
\begin{enumerate}
\item The functor $\widetilde{F}$ is self-adjoint. In particular, $\widetilde{F}$ is an exact functor.
\item We have $\widetilde{A}' \simeq \widetilde{F}(\widetilde{A})$.
\item We have $\widetilde{F}(\widetilde{M})\simeq \widetilde{M}\otimes_{\widetilde{A}} \widetilde{F}(\widetilde{A})$.
\item We have $\widetilde{\Phi}\circ F^Z \simeq \widetilde{F}\circ \widetilde{\Phi}$.
\end{enumerate}
\end{prop}
\begin{proof}
We already proved (1--3).
We have
\begin{multline*}
\widetilde{F}\circ \widetilde{\Phi}(M) = \Hom_{\widetilde{A}}(\widetilde{A}',\widetilde{\Phi}(M)) = \Hom_{\widetilde{A}}(\widetilde{\Phi}(\bigoplus_{y\in W}F^ZB(y)),\widetilde{\Phi}(M))\\
\simeq \Hom_Z(\bigoplus_{y\in W}F^ZB(y),M)
\simeq \Hom_Z(\bigoplus_{y\in W}B(y),F^ZM)
= \widetilde{\Phi}(F^Z(M)).
\end{multline*}
Hence we get (4).
\end{proof}

Now we discuss the restriction of $\widetilde{F}$ to the full-subcategory $\mathcal{O}$.
Fist we consider $F^Z = \theta_s^Z$.
For $M\in\ModCat{Z}$, $p\in S(V^*)$ induces a homomorphism $p\colon M\to M$.
Hence we have a homomorphism $\theta^Z_s(p)\colon \theta^Z_s(M)\to \theta^Z_s(M)$.
From the construction of $\theta_s^Z$, this map is equal to the action of $p\colon \theta^Z_s(M)\to \theta^Z_s(M)$.
Since $\widetilde{A}'$ is an $\widetilde{A}$-bimodule and $\widetilde{A}$ is a $S(V^*)$-algebra, $\widetilde{A}'$ is an $S(V^*)$-bimodule.
From the above argument, the left and right $S(V^*)$-module structure of $\widetilde{A}'$ coincide.
Hence the action of $S(V^*)$ on $\widetilde{\theta_s}(\widetilde{M}) = \Hom_{\widetilde{A}}(\widetilde{A}',\widetilde{M})$ coincides with the $S(V^*)$-action induced from that of $\widetilde{M}$.
In particular, if $\widetilde{M}$ is annihilated by $V^*$ (i.e., $\widetilde{M}\in \mathcal{O}$), then $\widetilde{\theta_s}(\widetilde{M})$ is also annihilated by $V^*$.
Hence $\widetilde{\theta_s}$ gives a functor from $\mathcal{O}$ to $\mathcal{O}$ and satisfies the similar properties in Proposition~\ref{prop:property of translation functors in widetilde O}.
We denote this functor by $\theta_s$.

In the case of $\varphi^Z_s$, the situation is bad.
In this case, a homomorphism $\varphi^Z_s(p)$ is not equal to $p$ for $p\in S(V^*)$ in general.
Hence $\widetilde{\varphi}_s$ dose not give a functor from $\mathcal{O}$ to $\mathcal{O}$.
Let $\varphi_s$ be the restriction of the functor $\widetilde{\varphi}_s$ to $\mathcal{O}$.
This is a functor from $\mathcal{O}$ to $\widetilde{\mathcal{O}}$.

\begin{rem}\label{rem:commutativity of tensor and translation functor}
By the same reason, we have $\theta_s(\widetilde{M}\otimes_{S(V^*)}\C) \simeq (\widetilde{\theta_s}(\widetilde{M}))\otimes_{S(V^*)}\C$ for $\widetilde{M}\in\widetilde{\mathcal{O}}$.
The corresponding statement for $\varphi_s$ is false in general.
\end{rem}

\subsection{Natural transformations}
We use the notation in the previous subsection.
We start with the following lemma.
\begin{lem}\label{lem:compatibility of natural transformations and adjointness in M}
For $M\in\mathcal{M}$, the natural transformation $M\to F^ZM$ is given by the self-adjointness of $F^Z$ and the natural transformation $F^ZM\to M$.
\end{lem}
\begin{proof}
We consider the case of $F^Z = \theta^Z_s$.
Using the functor $a_M$, we get the lemma in the case of $F^Z = \varphi_s$.

In this case, $F^ZM = Z\otimes_{Z^s}M$.
Since $(\Res_{Z^s},Z\otimes_{Z^s}\cdot)$, $(Z\otimes_{Z^s}\cdot,\Res_{Z^s})$ are adjoint pairs, we have
\[
	\Hom_Z(M,F^ZM)\simeq \Hom_{Z^s}(M,M)\simeq \Hom(F^ZM,M).
\]
The natural transformations $M\to F^ZM$ (resp.~$F^ZM\to M$) corresponds to $\Id\colon M\to M$ by the left (resp.~right) isomorphism.
Since these isomorphisms give a self-adjointness of $F^Z$, we get the lemma.
\end{proof}

Since $\widetilde{A}' = \widetilde{\Phi}(\bigoplus_{y\in W}(F^ZB(y)))$, we get a homomorphism $\sigma\colon \widetilde{A}\to \widetilde{A}'$ and $\sigma'\colon \widetilde{A}'\to \widetilde{A}$ from the natural transformation between $F^Z\colon\mathcal{M}\to \mathcal{M}$ and $\Id$.
Then $\sigma_{\widetilde{M}} = \Hom(\sigma,\widetilde{M})$ (resp.~$\sigma'_{\widetilde{M}} = \Hom(\sigma',\widetilde{M})$) gives a natural transformation $\sigma\colon \widetilde{F}\to \Id$ (resp.~$\sigma'\colon \Id\to \widetilde{F}$).

Since we have an isomorphism $\widetilde{F}(\widetilde{M})\simeq \widetilde{M}\otimes_{\widetilde{A}} \widetilde{A}'$, we can define another natural transformations by $\id_{\widetilde{M}}\otimes \sigma$ and $\id_{\widetilde{M}}\otimes\sigma'$.
\begin{prop}\label{prop:compatibility of natural transformations and adjointness in O}
We have $\sigma_{\widetilde{M}} = \id_{\widetilde{M}}\otimes\sigma'$ and $\sigma'_{\widetilde{M}} = \id_{\widetilde{M}}\otimes\sigma$.
Moreover, we have the following commutative diagram for $\widetilde{M},\widetilde{N}\in \widetilde{\mathcal{O}}$:
\[
\xymatrix{
\Hom(\widetilde{M},\widetilde{N})\ar[d]^{\Hom(\sigma_{\widetilde{M}},\widetilde{N})} \ar@{=}[r] & \Hom(\widetilde{M},\widetilde{N})\ar[d]^{\Hom(\widetilde{M},\sigma'_{\widetilde{N}})}\\
\Hom(\widetilde{F}\widetilde{M},\widetilde{N})\ar[d]^{\Hom(\sigma'_{\widetilde{M}},\widetilde{N})}\ar@{-}[r]^{\sim} & \Hom(\widetilde{M},\widetilde{F}\widetilde{N})\ar[d]^{\Hom(\widetilde{M},\sigma_{\widetilde{N}})}\\
\Hom(\widetilde{M},\widetilde{N})\ar@{=}[r] & \Hom(\widetilde{M},\widetilde{N}).
}
\]
\end{prop}
\begin{proof}
In this proof, we omit the grading of objects of $\mathcal{M}$.

First we prove the first claim for $\widetilde{M} = \widetilde{A}$.
Put $B = \bigoplus_{y\in W}B(y)$.
Recall that an isomorphism $\Hom(\widetilde{A}',\widetilde{A})\simeq \widetilde{A}'$ is induced from $\Hom_Z(F^ZB,B)\simeq \Hom_Z(B,F^ZB)$ and $\sigma$ (resp.~$\sigma'$) is induced from the natural transformation $\Id\to F^Z$ (resp.~$F^Z\to \Id$) in $\mathcal{M}$.
Hence we get the first claim for $\widetilde{M} = \widetilde{A}$ from the corresponding statement in $\mathcal{M}$ (Lemma~\ref{lem:compatibility of natural transformations and adjointness in M}).

To prove for a general $\widetilde{M}$, take a free resolution $\widetilde{N_1}\to \widetilde{N_0}\to \widetilde{M}\to 0$.
Since $\widetilde{F}$ is exact, we have $\Hom(\sigma,\widetilde{M}) = \Cok(\Hom(\sigma,\widetilde{N_1})\to \Hom(\sigma,\widetilde{N_0}))$.
Since $\widetilde{N_i}$ ($i = 0.1$) is free, we have $\Hom(\sigma,\widetilde{N_i}) = \id_{\widetilde{N_i}}\otimes \sigma'$.
Hence we have $\Hom(\sigma,\widetilde{M}) = \id_{\widetilde{M}}\otimes \sigma'$.
The same argument implies $\Hom(\sigma',\widetilde{M}) = \id_{\widetilde{M}}\otimes \sigma$.

We prove the second claim.
We only prove the commutativity of the lower square.
The same argument implies the proposition.
An isomorphism $\Hom(\widetilde{F}\widetilde{M},\widetilde{N}) \simeq \Hom(\widetilde{M},\widetilde{F}\widetilde{N})$ is equal to
\[
	\Hom(\widetilde{F}\widetilde{M},\widetilde{N})\simeq \Hom(\widetilde{M}\otimes_{\widetilde{A}}\widetilde{A}',\widetilde{N})\simeq \Hom(\widetilde{M},\Hom_{\widetilde{A}}(\widetilde{A}',\widetilde{N})) = \Hom(\widetilde{M},\widetilde{F}\widetilde{N}).
\]
For $f\in \Hom(\widetilde{F}\widetilde{M},\widetilde{N}) = \Hom(\widetilde{M}\otimes_{\widetilde{A}}\widetilde{A}',\widetilde{N})$, an image  of $f$ under $\Hom(\widetilde{F}\widetilde{M},\widetilde{N})\simeq \Hom(\widetilde{M},\widetilde{F}\widetilde{N})\to \Hom(\widetilde{M},\widetilde{N})$ is given by $m\mapsto f(m\otimes \sigma(1))$, namely, an image of $f$ under the map $\Hom(\id_{\widetilde{M}}\otimes \sigma,\widetilde{N})$.
We get the proposition from the first claim.
\end{proof}

\begin{thm}\label{thm:theta and varphi commute, in O}
Let $s,t$ be simple reflections.
The functors $\widetilde{\theta_t}$ and $\widetilde{\varphi_s}$ from $\widetilde{\mathcal{O}}$ to $\widetilde{\mathcal{O}}$ commute with each other.
Moreover, the natural transformation $\widetilde{\theta_t} \to \widetilde{\varphi_s}\widetilde{\theta_t}$ (resp.~$\widetilde{\varphi_s} \to \widetilde{\theta_t}\widetilde{\varphi_s}$, $\widetilde{\varphi_s}\widetilde{\theta_t}\to \widetilde{\theta_t}$, $\widetilde{\theta_t}\widetilde{\varphi_s}\to \widetilde{\varphi_s}$) can be identified with $\widetilde{\theta_t}(\Id\to \widetilde{\varphi_s})$ (resp.~$\widetilde{\varphi_s} (\Id\to\widetilde{\theta_t})$, $\widetilde{\theta_t}(\widetilde{\varphi_s}\to \Id)$, $\widetilde{\varphi_s}(\widetilde{\theta_t}\to\Id)$).
\end{thm}
\begin{proof}
Since $\widetilde{\varphi_s}(\widetilde{M})\simeq \widetilde{M}\otimes_{\widetilde{A}}\widetilde{\varphi_s}(\widetilde{A})$ and $\widetilde{\theta_t}(\widetilde{M})\simeq \widetilde{M}\otimes_{\widetilde{A}}\widetilde{\theta_t}(\widetilde{A})$, we may assume that $\widetilde{M} = \widetilde{A}$.
In this case, the theorem follows from the corresponding statement in $\mathcal{M}$, namely, Proposition~\ref{prop:theta and varphi commute}.
\end{proof}

\subsection{Translation of projective modules and simple modules}
\begin{thm}\label{thm:translation of projective modules}
\begin{enumerate}
\item If $xs < x$, then $\widetilde{\theta_s} \widetilde{P}(x) = \widetilde{P}(x)^{\oplus 2}$ and $\theta_sP(x) = P(x)^{\oplus 2}$.\label{thm:translation of projective modules:theta_s simple}
\item If $xs > x$, then $\widetilde{\theta_s} \widetilde{P}(x) = \widetilde{P}(xs)\oplus\bigoplus_{y < x,\ ys < y}\widetilde{P}(y)^{m_y}$ and $\theta_s P(x) = P(xs)\oplus\bigoplus_{y < x,\ ys < y}P(y)^{m_y}$for some $m_y\in\Z_{\ge 0}$.\label{thm:translation of projective modules:theta_s complicated}
\item $\theta_s L(x) = 0$ if and only if $xs > x$.\label{thm:translation of projective modules:theta_s vanishing of L(x)}
\item If $sx < x$, then $\widetilde{\varphi_s} \widetilde{P}(x) = \widetilde{P}(x)^{\oplus 2}$.\label{thm:translation of projective modules:varphi_s simple}
\item If $sx > x$, then $\widetilde{\varphi_s} \widetilde{P}(x) = \widetilde{P}(sx)\oplus\bigoplus_{y < x,\ sy < y}\widetilde{P}(y)^{m_y}$ for some $m_y\in\Z_{\ge 0}$.\label{thm:translation of projective modules:varphi_s complicated}
\item $\varphi_s L(x) = 0$ if and only if $sx > x$.\label{thm:translation of projective modules:varphi_s vanishing of L(x)}
\end{enumerate}
\end{thm}
\begin{proof}
The first statement of (1) and (2) follows from Proposition~\ref{prop:translation of projective Z-module} and Proposition~\ref{prop:property of translation functors in widetilde O}.
We get the second statement of (1) (2) tensoring $\C$ to the first statement of (1) (2), respectively (see Remark~\ref{rem:commutativity of tensor and translation functor}).

From (1) and (2), we have $\theta_s A = \bigoplus_{ys < y}P(y)^{n_y}$ for some $n_y\ge 2$.
Put $n_y = 0$ for $ys > y$.
Then we have
\begin{multline*}
\dim\theta_s L(x) = \dim\Hom_{\widetilde{A}}(\widetilde{A},\theta_sL(x)) = \dim\Hom_{\widetilde{A}}(\widetilde{\theta_s} \widetilde{A},L(x))\\
=\dim\Hom_A\left(\bigoplus_y \widetilde{P}(y)^{n_y},L(x)\right) = n_y.
\end{multline*}
The proposition follows.

(4), (5) and (6) follow from the same argument.
\end{proof}

\section{Zuckerman functor}\label{sec:Zuckerman functor}
\subsection{Definition and commutativity with translation functors}
Fix a simple reflection $s$.
Let $\mathcal{O}_s$ be a full-subcategory of $\mathcal{O}$ consisting of a module $M$ such that $\Hom_A(P(x),M) = 0$ for all $sx < x$.
Let $\iota_s\colon \mathcal{O}_s\to \mathcal{O}$ be the inclusion functor.
Then $\iota_s$ has the left adjoint functor $\widetilde{\tau}_s$.
It is defined by 
\[
	\widetilde{\tau}_s(M) = M/M'
\]
where
\[
	M' = \bigcap_{\varphi\colon M\to M_1,\ M_1\in\mathcal{O}_s}\Ker\varphi.
\]
Since $\widetilde{\tau}_s$ has the right adjoint functor $\iota_s$, $\widetilde{\tau}_s$ is a right exact functor.
Put $\tau_s = \iota_s\widetilde{\tau}_s$.

\begin{lem}\label{lem:characterization of O_s and translation preserve O_s}
Let $s$ be a simple reflection.
For $M\in\mathcal{O}$, $M\in\mathcal{O}_s$ if and only if $\varphi_sM = 0$.
In particular, $\theta_t$ preserves the category $\mathcal{O}_s$ for a simple reflection $t$.
\end{lem}
\begin{proof}
From Theorem~\ref{thm:translation of projective modules}, we have $\widetilde{\varphi_s}\widetilde{A} = \bigoplus_{sy < y}\widetilde{P}(y)^{m_y}$ for some $m_y\ge 2$.
Hence, if $M\in\mathcal{O}_s$, then $\varphi_sM = \Hom_{\widetilde{A}}(\widetilde{A},\varphi_sM) = \Hom_{\widetilde{A}}(\widetilde{\varphi_s}\widetilde{A},M) = 0$.

If $M\not\in\mathcal{O}_s$, then $\Hom(\widetilde{P}(x),M) = \Hom(P(x),M) \ne 0$ for some $x\in W$ such that $sx < x$.
Hence $\Hom(\widetilde{P}(x),\varphi_sM) = \Hom(\widetilde{\varphi_s}\widetilde{P}(x),M) = \Hom(\widetilde{P}(x)^{\oplus 2},M) \ne 0$.
Therefore, $\varphi_sM \ne 0$.

Take $M\in \mathcal{O}_s$.
Then, by Theorem~\ref{thm:theta and varphi commute, in O}, $\varphi_s\theta_tM = \widetilde{\theta_t}\varphi_sM = 0$.
Hence $\theta_tM\in\mathcal{O}_s$.
\end{proof}

\begin{prop}\label{prop:translation and tau commute}
The functors $\tau_s$ and $\theta_t$ commute with each other for simple reflections $s,t$.
\end{prop}
\begin{proof}
From Lemma~\ref{lem:characterization of O_s and translation preserve O_s}, the functor $\theta_t$ induces a self-adjoint functor from $\mathcal{O}_s$ to $\mathcal{O}_s$.
We denote this functor by $\theta_t'$.
Obviously, we have $\theta_t \iota_s \simeq \iota_s \theta'_t$.
Taking the left adjoint functor of the both sides, we get $\widetilde{\tau}_s \theta_t\simeq \theta'_t \widetilde{\tau}_s$.
Hence we get $\theta_t \tau_s = \theta_t \iota_s \widetilde{\tau}_s\simeq \iota_s\theta'_t\widetilde{\tau}_s \simeq \iota_s\widetilde{\tau}_s\theta_t = \tau_s\theta_t$.
\end{proof}

\subsection{Translation of Verma modules}
We consider $\varphi_s M(x)$.
We start with two lemmas.
\begin{lem}
Let $\{M_\lambda\}$ be a family of $S(V^*)$-modules.
Then we have an isomorphism $(\prod_\lambda M_\lambda)\otimes_{S(V^*)}\C \simeq \prod_\lambda (M_\lambda\otimes_{S(V^*)}\C)$.
\end{lem}
\begin{proof}
Since $M\otimes_{S(V^*)}\C = M/V^*M$ for an $S(V^*)$-module $M$, it is sufficient to prove that $V^*(\prod_\lambda M_\lambda) = \prod_\lambda (V^*M_\lambda)$.
Notice that $V^*$ is finite-dimensional.
Let $v_1,\dots,v_r$ be a basis of $V^*$.
Then $V^*(\prod_\lambda M_\lambda) = \sum_i v_i(\prod_\lambda M_\lambda) = \sum_i\prod_\lambda v_iM_\lambda = \prod_\lambda (V^*M_\lambda)$
\end{proof}

\begin{lem}
Let $M_1\to M_2\to M_3$ be a sequence in $\mathcal{M}$.
If $\Hom_Z(B(y),M_1)\otimes_{S(V^*)}\C\to \Hom_Z(B(y),M_2)\otimes_{S(V^*)}\C\to \Hom_Z(B(y),M_3)\otimes_{S(V^*)}\C$ is exact for all $y$, then $\Phi(M_1)\to \Phi(M_2)\to \Phi(M_3)$ is exact.
\end{lem}
\begin{proof}
From the previous lemma,
\begin{align*}
	\prod_{y\in W}(\Hom_Z(B(y),M)\otimes_{S(V^*)}\C) & \simeq 
	\left(\prod_{y\in W}\Hom_Z(B(y),M)\right)\otimes_{S(V^*)}\C\\
	& \simeq \Hom_Z\left(\bigoplus_{y\in W}B(y),M\right)\otimes_{S(V^*)}\C\\
	& = \Phi(M).
\end{align*}
We get the lemma.
\end{proof}

\begin{prop}\label{prop:translation of M(e), etc...}
Let $s$ be a simple reflection and $x\in W$ such that $sx > x$.
\begin{enumerate}
\item We have an exact sequence $0\to M(x)\to \Phi(\varphi_s^ZV(sx))\to M(sx)\to 0$, here the map $\Phi(\varphi_s^ZV(sx))\to M(sx)$ is the canonical map.
\item We have an exact sequence $0\to M(x)\to \varphi_s M(sx)\to M(sx)\to 0$, here the map $\varphi_s M(sx)\to M(sx)$ is the canonical map.
\item We have an isomorphism $\widetilde{\varphi_s} \widetilde{M}(sx)\simeq \widetilde{\varphi_s} \widetilde{M}(x)$ and the map $M(x)\to \varphi_sM(sx)$ in (1) and $M(x)\to \widetilde{\varphi_s}\widetilde{M}(sx)\otimes_{S(V^*)}\C$ is induced from the canonical map $\widetilde{M}(x)\to \widetilde{\varphi_s} \widetilde{M}(x)$.
\item For a $Z$-module $M$, the composition of the maps $\Phi(M)\to \varphi_s \Phi(M)\to \Phi(M)$ is equal to $0$.
\item We have an inclusion $M(sx)\to M(x)$.
\end{enumerate}
\end{prop}
\begin{proof}
Set $\alpha = \alpha_s$.

(1)
Put $\mathscr{M} = \mathscr{L}(\varphi_s V(sx))$.
By \cite[Lemma~5.4]{MR2395170}, we have
\begin{align*}
\mathscr{M}_y & = 
\begin{cases}
S(V^*)\langle -1\rangle & (\text{$y = x$ or $sx$}),\\
0 & (\text{otherwise}),
\end{cases}\\
\mathscr{M}_E & = 
\begin{cases}
S(V^*)/\alpha S(V^*)\langle -1\rangle & (\text{$h_{\mathcal{G}}(E) = x$, $t_{\mathcal{G}}(E) = sx$}),\\
0 & (\text{otherwise}).
\end{cases}
\end{align*}
Hence we get an exact sequence $V(x)\langle -1\rangle\to \varphi_s V(sx)\langle 1\rangle\to V(sx)$ (cf.~\cite[3.4]{MR2395170}).
This implies an exact sequence
\[
	0\to \Hom_Z(B(y),V(x))\to\Hom_Z(B(y),\varphi_s^Z V(sx))\to \Hom_Z(B(y),V(sx))\to 0
\]
for all $y\in W$.
Since $\Hom_Z(B(y),V(sx))\simeq \Hom_{S(V^*)}(\mathscr{B}(y)_{sx},S(V^*))$ and $\mathscr{B}(y)_{sx}$ is free, we have that $\Hom_Z(B(y),V(sx))$ is free.
Hence we get an exact sequence,
\begin{multline*}
0\to \Hom_Z(B(y),V(x))\otimes_{S(V^*)}\C\to\Hom_Z(B(y),\varphi_s^Z V(sx))\otimes_{S(V^*)}\C\\
\to \Hom_Z(B(y),V(sx))\otimes_{S(V^*)}\C\to 0
\end{multline*}
for all $y\in W$.
From the previous lemma, we get (1).

(2)
For $\widetilde{M}\in\widetilde{\mathcal{O}}$, we define a new $S(V^*)$-module structure on $\widetilde{\varphi_s}(\widetilde{M})$ as follows.
The action of $p\in S(V^*)$ is given by $\varphi_s(p)$, here $p\colon \widetilde{M}\to \widetilde{M}$ is a $S(V^*)$-action on $\widetilde{M}$.
Then, in general, this action is different from the original $S(V^*)$-action (the action induced from the action of $\widetilde{A}$).
When we consider this $S(V^*)$-module structure, we denote $C(\widetilde{\varphi_s}(\widetilde{M}))$ instead of $\widetilde{\varphi_s}(\widetilde{M})$.
By the definition, we get $C(\widetilde{\varphi_s}(\widetilde{M}))\otimes_{S(V^*)}\C = C(\widetilde{\varphi_s}(\widetilde{M}\otimes_{S(V^*)}\C))$.
We define the $S(V^*)$-module structure on $\Hom_Z(B(y),\varphi_s^ZV(sx))$ by the same way, and denote the resulting $S(V^*)$-module by $C^Z(\Hom_Z(B(y),\varphi_s^ZV(sx)))$.
We have $C^Z(\Hom_Z(\bigoplus_{y\in W}B(y),\varphi_s^ZV(sx))) = C(\widetilde{\varphi_s}\widetilde{M}(sx))$.
Moreover, from the same argument in (1), we have an exact sequence
\[
	0\to \Hom_Z(B(y),V(x))\to C(\Hom_Z(B(y),\varphi_s^Z V(sx)))\to \Hom_Z(B(y),V(sx))\to 0
\]
for all $y\in W$.
Tensoring with $\C$, we get (2).

(3)
Both $V(x)$ and $V(sx)$ are isomorphic to $S(V^*)$ as an $S(V^*)$-module.
Let $z = (z_w)_w\in Z\subset \prod_{w\in W}S(V^*)$ and assume that $z\in Z^{r_s}$.
Then we have $z_x = s(z_{sx})$.
Hence the action of $z$ on $V(x)$ is given by the multiplication of $z_x$, while the action of $z$ on $V(sx)$ is given by the multiplication of $z_{sx} = s(z_x)$.
Hence $S(V^*) \simeq V(x)\to V(sx)\simeq S(V^*)$ given by $p\mapsto s(p)$ is an isomorphism as $Z^{r_s}$-modules.
Hence $\Res_{Z^{r_s}} V(x)\simeq \Res_{Z^{r_s}}V(sx)$.
Therefore, $\varphi_s^Z V(x)\simeq \varphi_s^Z V(sx)$.
Hence we get $\widetilde{\varphi_s} \widetilde{M}(x)\simeq \widetilde{\varphi_s} \widetilde{M}(sx)$.
It is easy to see that the canonical map $M(x)\to \varphi_sM(x)$ is equal to the map we give in (1) and (2).

(4)
The composition of the maps $M\to Z\otimes_{Z^{r_s}}M\to M$ is given by $m\mapsto 2\alpha m$.
So the map $\Hom_Z(B,M)\to \Hom_Z(B,\varphi_s^ZM)\to \Hom_Z(B,M)$ is given by $f\mapsto 2\alpha f$.
If we tensor $\C$ over $S(V^*)$, this map becomes $0$.

(5)
This is a consequence of (1) and (4).
\end{proof}

\subsection{Duality of Zuckerman functor}
\begin{lem}\label{lem:calculation of tau_s(M(e))}
Let $f\colon M(s)\to M(e)$ be an injective map.
Then we have $\tau_s(M(e)) = M(e)/f(M(s))$.
\end{lem}
\begin{proof}
Put $M = \Ker(M(e)\to \tau_sM(e))$.
If $sx > x$, we have $\mathscr{B}(x)_e = \mathscr{B}(x)_s$ by Lemma~\ref{lem:translation out of projectives in Z-mod} and \cite[Lemma~5.4]{MR2395170}.
Hence 
\begin{multline*}
	\rank \Hom_Z(B(x),V(e)) = \rank \Hom_{S(V^*)}(\mathscr{B}(x)_e,S(V^*))\\
	= \rank \Hom_{S(V^*)}(\mathscr{B}(x)_s,S(V^*)) = \rank\Hom_Z(B(x),V(s)).
\end{multline*}
This implies $\dim\Hom_A(P(x),M(e)) = \dim\Hom_A(P(x),M(s))$.
Therefore, we get $\Hom_A(P(x),M(e)/f(M(s))) = 0$.
Hence $M\subset f(M(s))$.
Since $f(M(s))\simeq M(s)$ has the unique irreducible quotient $L(s)$, we have $M = f(M(s))$.
\end{proof}

The module $\tau_s(A)$ is, of course, a right $A$-module.
Using $A\simeq \End_A(A,A)\to \End_A(\tau_s(A),\tau_s(A))$, we also regard $\tau_s(A)$ as a left $A$-module.
By the same argument, $\varphi_s(A)$ is a left $A$-module and right $\widetilde{A}$-module.

\begin{thm}\label{thm:projective resolution of tau_s(A)}
We have the following exact sequences, here all maps are canonical maps.
\begin{enumerate}
\item $0\to A\to \varphi_s A\to A\to \tau_s A\to 0$ as left $A$- and right $\widetilde{A}$-modules.
\item $0\to A\to (\widetilde{\varphi_s}\widetilde{A})\otimes_{S(V^*)}\C\to A\to\tau_s A\to 0$ as left $\widetilde{A}$- and right $A$-modules. 
\end{enumerate}
\end{thm}
\begin{proof}
We only prove (1).
The same argument implies (2).

We prove the exactness of $0\to P(x)\to \varphi_s P(x)\to P(x)\to \tau_s P(x)\to 0$ by induction on $\ell(x)$.

First assume that $x = e$.
Then $P(e) = M(e)$.
By Proposition~\ref{prop:translation of M(e), etc...} (1) and (3), $0\to M(e)\to \varphi_s M(e)$ is exact and its cokernel is isomorphic to $M(s)$.
From Lemma~\ref{lem:calculation of tau_s(M(e))}, we have an exact sequence $0\to M(s)\to M(e)\to \tau_sM(e)\to 0$.
Hence $0\to M(e)\to \varphi_s M(e)\to M(e)\to \tau_sM(e)\to 0$ is exact.

Assume that $x > e$ and take a simple reflection $t$ such that $xt < x$.
Then by inductive hypothesis, the sequence $0\to P(xt)\to \varphi_s P(xt)\to P(xt)\to \tau_s P(xt) \to 0$ is exact.
By Theorem~\ref{thm:theta and varphi commute, in O} and Proposition~\ref{prop:translation and tau commute}, we get the exact sequence $0\to \theta_t P(xt)\to \varphi_s\theta_t P(xt)\to \theta_t P(xt)\to \tau_s \theta_tP(xt)\to 0$.
Since $P(x)$ is a direct summand of $\theta_t P(xt)$, we get the theorem.
\end{proof}

\begin{lem}\label{lem:expression of varphi_s}
For $M\in\mathcal{O}$, we have the following.
\begin{enumerate}
\item We have $\varphi_s(M)\simeq M\otimes_A\varphi_s(A)$. Hence $\varphi_s(A)$ is a flat left $A$-module
\item We have $\Hom_A(\widetilde{\varphi_s}(\widetilde{A})\otimes_{S(V^*)}\C,M) \simeq \varphi_s(M)$.
Hence $\widetilde{\varphi_s}(\widetilde{A})\otimes_{S(V^*)}\C$ is a projective right $A$-module.
\end{enumerate}
\end{lem}
\begin{proof}
(1) follows from Lemma~\ref{lem:realization of a right exact functor as a tensor product}.
(2) is proved by the following equation:
\begin{multline*}
\Hom_A(\widetilde{\varphi_s}(\widetilde{A})\otimes_{S(V^*)}\C,M) = \Hom_{\widetilde{A}}(\widetilde{\varphi_s}(\widetilde{A}),M)\\ \simeq \Hom_{\widetilde{A}}(\widetilde{A},\widetilde{\varphi_s}M) \simeq \widetilde{\varphi_s}(M) = \varphi_s(M)
\end{multline*}
\end{proof}

Define a functor $\tau'_s\colon \mathcal{O}\to \mathcal{O}$ by $\tau'_s(M) = \Hom_A(\tau_s(A),M)$.
Since $\tau_s(M)\simeq M\otimes_A\tau_s(A)$, this functor is the right adjoint functor of $\tau_s$.
Let $L\tau_s$ be the left derived functor of $\tau_s$, $R\tau'_s$ the right derived functor of $\tau'_s$, $D^b(\mathcal{O})$ the bounded derived category of $\mathcal{O}$.

\begin{lem}
We have $R\tau'_s(A)[2] \simeq \tau_s(A)$ as $A$-bimodules.
\end{lem}
\begin{proof}
We prove that $R^i\tau'_s(A) = 0$ for $i\ne 2$ and $R^2\tau'_s(A) = \tau_s(A)$.
Let $k\colon D(\mathcal{O})\to D(\widetilde{\mathcal{O}})$ be the functor induced from the inclusion functor $\mathcal{O}\to\widetilde{\mathcal{O}}$.
It is sufficient to consider $k(R\tau'_s(A))$ since $k$ is an exact functor.
We calculate $R\Hom_A(\tau_s(A),M)$ using the projective resolution in Theorem~\ref{thm:projective resolution of tau_s(A)} (2).
(The reason why we calculate $k(R\tau'_s(A))$ is that a projective resolution in Theorem~\ref{thm:projective resolution of tau_s(A)} is an exact sequence not of $A$-bimodules but of left $\widetilde{A}$- and right $A$-modules.)

From Theorem~\ref{thm:projective resolution of tau_s(A)} (2), $R\Hom_A(\tau_s(A),A)$ is given by the complex 
\[
	\cdots\to \Hom_A(A,A)\to \Hom_A(\widetilde{\varphi_s}(\widetilde{A})\otimes_{S(V^*)}\C,A)\to \Hom_A(A,A)\to\cdots.
\]
By Lemma~\ref{lem:expression of varphi_s}, this complex is 
\[
	\cdots\to A\to \varphi_s(A)\to A\to \cdots.
\]
From Theorem~\ref{thm:projective resolution of tau_s(A)} (1), this complex is equal to $\tau_s(A)[-2]$.
\end{proof}

\begin{thm}\label{thm:duality of Zuckerman functor}
Let $s$ be a simple reflection.
\begin{enumerate}
\item We have $L^i\tau_s(M) = 0$ for $i > 2$ and $M\in\mathcal{O}$.
Hence $L\tau_s$ gives a functor from $D^b(\mathcal{O})$ to $D^b(\mathcal{O})$.\label{thm:duality of Zuckerman functor:Vanishing of higher cohomology}
\item The functor $L\tau_s[-1]$ is self-adjoint.
More generally, for $M,N\in D^b(\mathcal{O})$, we have $R\Hom(L\tau_s M[-1],N) = R\Hom(M,L\tau_s N[-1])$.\label{thm:duality of Zuckerman functor:duality}
\end{enumerate}
\end{thm}
\begin{proof}
Let $k\colon D(\mathcal{O})\to D(\widetilde{\mathcal{O}})$ be the functor induced from the inclusion functor $\mathcal{O}\to\widetilde{\mathcal{O}}$.
We prove that $H^i(k(L\tau_s(M))) = 0$ for $i > 2$.
By Theorem~\ref{thm:projective resolution of tau_s(A)} and isomorphism $\tau_s(M)\simeq \tau_s(A)\otimes_AM$, $k(L\tau_s(M))$ is given by the complex $(0\to M\to M\otimes_A\varphi_s(A)\to M\to 0)$.
From this description, we get (1).

By the definition, $\tau'_s$ is the right adjoint functor of $\tau_s$.
Hence we have an isomorphism $R\Hom(L\tau_s M,N)\simeq R\Hom(M,R\tau'_sN)$.
To prove (2), it is sufficient to prove that $R\tau'_s[2] = L\tau_s$.
Since $L\tau_s(M) \simeq M\otimes_A^L\tau_s(A)$, we have 
\begin{multline*}
(L\tau_s)^2(M) \simeq M\otimes_A^L\tau_s(A)\otimes_A^L\tau_s(A) \simeq M\otimes_L^AL\tau_s(\tau_s(A))\\\simeq M\otimes_A^LL\tau_s(R\tau'_s(A))[2]\to M\otimes_A^LA[2] = M[2],
\end{multline*}
here the last map is induced from the adjointness of $L\tau_s$ and $R\tau'_s$.
Hence using the adjointness again, we get the map $L\tau_s(M)\to R\tau'_s(M)[2]$.
If $A = M$, then this homomorphism is an isomorphism.
For a general $M$, taking a projective resolution, we can prove that the homomorphism is an isomorphism.
\end{proof}

\section{The functors $T_s$ and $C_s$}\label{sec:THe functors T_s anc C_s}
\subsection{Definition and adjointness}
Let $s$ be a simple reflection.
Define a functor $\widetilde{T_s}\colon\widetilde{\mathcal{O}}\to \widetilde{\mathcal{O}}$ by $\widetilde{T_s}(\widetilde{M}) = \Cok(\widetilde{M}\to \widetilde{\varphi_s}(\widetilde{M}))$.
The exactness of $\widetilde{\varphi_s}$ implies that $\widetilde{T_s}$ is right exact.
\begin{lem}\label{lem:T_s is well-defined}
For $p\in S(V^*)$ and $\widetilde{M}\in\widetilde{\mathcal{O}}$, we have $s(p) = \widetilde{T_s}(p)\colon \widetilde{T_s}(\widetilde{M})\to \widetilde{T_s}(\widetilde{M})$.
In particular, we have $\widetilde{T_s}(\mathcal{O})\subset \mathcal{O}$.

\end{lem}
\begin{proof}
Since $\widetilde{T_s}$ is right exact, we have $\widetilde{T_s}(M) \simeq M\otimes_A\widetilde{T_s}(A)$.
Hence we may assume that $M = A$.
Set $B = \bigoplus_{y\in W}B(y)$.
Then we have 
\begin{align*}
\varphi_s(A) & = \Hom_{\widetilde{A}}(\Phi(\varphi_s^Z(B)),A)\\
& = \Hom_{\widetilde{A}}(\Hom_Z(B,\varphi_s^Z(B)),A)\\
& \simeq \Hom_{\widetilde{A}}(\Hom_Z(\varphi_s^Z(B),B),A)\\
& = \Hom_{\widetilde{A}}(\Hom_Z(Z\otimes_{Z^{r_s}}B,B),A).
\end{align*}
Take $f\in \Hom_Z(Z\otimes_{Z^{r_s}}B,B)$, $z\in Z$ and $b\in B$.
Then $p\in S(V^*)$ can acts on $f$ by two ways.
The first way is induced from the right $\widetilde{A}$-module structure, namely, $f\mapsto ((z\otimes b)\mapsto f(z\otimes pb))$, this induces a homomorphism $p\colon \varphi_s(A)\to \varphi_s(A)$.
The second way is induced from the left $\widetilde{A}$-module structure, namely, $f\mapsto ((z\otimes b)\mapsto pf(z\otimes b))$, this induces a homomorphism $\varphi_s(p)\colon \varphi_s(A)\to \varphi_s(A)$.
We denote the first action by $f\mapsto pf$ and section action by $f\mapsto p\cdot f$.
For $p\in S(V^*)\subset Z$, we have $r_s(p) = s(p)$.
Hence if $p\in S(V^*)^s$, then we have $p\in Z^{r_s}$.
So, in this case, we get $pf = p\cdot f$.
Hence $p = \widetilde{\varphi_s}(p)$.
This implies $p = \widetilde{T_s}(p)$.

Set $\alpha = \alpha_s$.
Since $S(V^*) = S(V^*)^s\oplus \alpha S(V^*)^s$, it is sufficient to prove that $T_s(\alpha) = -\alpha$.
The natural transformation $A\to \varphi_s(A)$ is induced from $B\to Z\otimes_{Z^{r_s}}B$ and it is given by $b\mapsto (\alpha\otimes b + 1\otimes \alpha b)$ (Remark~\ref{rem:explicit discription of the natural transformation of theta_s}).
Hence $A \simeq \Hom_{\widetilde{A}}(\widetilde{A},A) = \Hom_{\widetilde{A}}(\Hom_Z(B,B),A)\to \varphi_s(A) = \Hom_{\widetilde{A}}(\Hom_Z(Z\otimes_{Z^{r_s}}B,B),A)$ is given by 
\[
	a\mapsto (f\mapsto a(b\mapsto f(\alpha \otimes b + 1\otimes \alpha b))),
\]
where $a\in A \simeq \Hom_{\widetilde{A}}(\Hom_Z(B,B),A)$, $f\in \Hom(Z\otimes_{Z^{r_s}}B,B)$ and $b\in B$.

Take $a'\in \Hom_{\widetilde{A}}(\Hom_Z(Z\otimes_{Z^{r_s}}B,B),A)$ and define $a\in \Hom_{\widetilde{A}}(\Hom_Z(B,B),A)$ by
\[
	\Hom_Z(B,B)\ni g\mapsto (a'(z\otimes b\mapsto g(zb))).
\]
Since $B\to Z\otimes_{Z^{r_s}}B$; $b\mapsto (\alpha\otimes b + 1\otimes \alpha b)$ is a $Z$-module homomorphism, we have $(\alpha\otimes zb + 1\otimes \alpha zb) = (z\alpha\otimes b + z\otimes \alpha b)$.
Hence the image of $a$ in $\Hom_{\widetilde{A}}(\Hom_Z(Z\otimes_{Z^{r_s}}B,B),A)$ is 
\begin{multline*}
f\mapsto a'(z\otimes b\mapsto f(\alpha\otimes zb + 1\otimes \alpha zb) = f(\alpha z\otimes b + z\otimes \alpha b))\\
= a'(\alpha f + \alpha\cdot f)
= (\alpha a' + \widetilde{\varphi_s}(\alpha)a')(f).
\end{multline*}
Therefore, we get $\alpha + \widetilde{T_s}(\alpha) = 0$.
\end{proof}

We denote the restriction of $\widetilde{T_s}$ on $\mathcal{O}$ by $T_s$.
This gives a functor $T_s\colon \mathcal{O}\to \mathcal{O}$.
We define the functor $\widetilde{C_s}\colon\widetilde{\mathcal{O}}\to\widetilde{\mathcal{O}}$ by $\widetilde{C_s}(\widetilde{M}) = \Ker(\widetilde{\varphi_s}(\widetilde{M})\to \widetilde{M})$.
\begin{prop}\label{prop:adjointness of T_s and C_s, tilde}
The functor $\widetilde{C_s}$ is the right adjoint functor of $\widetilde{T_s}$.
\end{prop}
\begin{proof}
From Proposition~\ref{prop:compatibility of natural transformations and adjointness in O}, we get the following commutative diagram:
\[
\xymatrix{
0\ar[r] & \Hom(\widetilde{M},\widetilde{C_s}\widetilde{N})\ar[r] & \Hom(\widetilde{M},\widetilde{\varphi_s}\widetilde{N})\ar[r]\ar@{-}[d]^{\wr} & \Hom(\widetilde{M},\widetilde{N})\ar@{=}[d]\\
0\ar[r] & \Hom(\widetilde{T_s}\widetilde{M},\widetilde{N})\ar[r] & \Hom(\widetilde{\varphi_s}\widetilde{M},\widetilde{N})\ar[r] & \Hom(\widetilde{M},\widetilde{N}).
}
\]
We get the Proposition.
\end{proof}
In particular, for $M\in \mathcal{O}$, we have
\[
\widetilde{C_s}(M) \simeq \Hom_{\widetilde{A}}(\widetilde{A},\widetilde{C_s}(M)) \simeq \Hom_{\widetilde{A}}(T_s(\widetilde{A}),M)
\simeq \Hom_{\widetilde{A}}(T_s(\widetilde{A})/V^*T_s(\widetilde{A}),M).
\]
From Lemma~\ref{lem:T_s is well-defined}, we have $T_s(\widetilde{A})/V^*T_s(\widetilde{A})\simeq T_s(\widetilde{A}/V^*\widetilde{A}) = T_s(A)$.
Hence we get $\widetilde{C_s}(M) = \Hom_{\widetilde{A}}(T_s(A),M)$.
From this formula, we get $\widetilde{C_s}(M)\in\mathcal{O}$.
Hence $\widetilde{C_s}$ defines the functor $C_s\colon\mathcal{O}\to\mathcal{O}$.
From Proposition~\ref{prop:adjointness of T_s and C_s, tilde}, we get the following theorem.
\begin{thm}\label{thm:adjointness of T_s and C_s}
The functor $C_s$ is the right adjoint functor of $T_s$.
\end{thm}
Finally, we prove the following lemma.
This lemma assures the existence of a natural translation $T_s\to \Id$ and $\Id\to C_s$.
\begin{lem}
For $M\in\mathcal{O}$, the composition of the maps $M\to \varphi_s(M)\to M$ is zero.
\end{lem}
\begin{proof}
From Proposition~\ref{prop:property of translation functors in widetilde O}, $\varphi_s(M) = M\otimes_A \varphi_s(A)$.
Hence we may assume that $M = A = \Phi(\bigoplus_{x\in W}B(x))$.
By Lemma~\ref{lem:calculation of tau_s(M(e))}, we get the lemma.
\end{proof}

\subsection{Homological properties}
\begin{prop}\label{prop:LT_s RC_s and Ltau_s}
Let $s$ be a simple reflection.
\begin{enumerate}
\item We have $L^iT_s = 0$ for $i > 1$.
Hence $LT_s$ gives a functor $D^b(\mathcal{O})\to D^b(\mathcal{O})$.\label{prop:LT_s RC_s and Ltau_s:finiteness of LT_s}
\item We have a distinguished triangle $LT_s\to \id \to L\tau_s\xrightarrow{+1}$.\label{prop:LT_s RC_s and Ltau_s:distinguished triangle of LT_s and Ltau_s}
\item We have $R^iC_s = 0$ for $i > 1$.
Hence $RC_s$ gives a functor $D^b(\mathcal{O})\to D^b(\mathcal{O})$.\label{prop:LT_s RC_s and Ltau_s:finiteness of RC_s}
\item We have a distinguished triangle $L\tau_s[-2]\to \id\to RC_s\xrightarrow{+1}$.\label{prop:LT_s RC_s and Ltau_s:distinguished triangle of RC_s and Ltau_s}
\item We have $L^1T_sM = \Ker(M\to \varphi_sM)$ and $R^1C_sM = \Cok(\varphi_sM \to M)$.\label{prop:LT_s RC_s and Ltau_s:cohomology of T_s and C_s}
\end{enumerate}
\end{prop}
\begin{proof}
(\ref{prop:LT_s RC_s and Ltau_s:finiteness of LT_s}) follows from (\ref{prop:LT_s RC_s and Ltau_s:distinguished triangle of LT_s and Ltau_s}) and Theorem~\ref{thm:duality of Zuckerman functor} (\ref{thm:duality of Zuckerman functor:Vanishing of higher cohomology}).
By Theorem~\ref{thm:projective resolution of tau_s(A)}, we have $0\to T_s(A)\to A\to \tau_s(A)\to 0$.
Since $T_s$ and $\tau_s$ are right exact, we have $T_s(M) = M\otimes_A T_s(A)$ and $\tau_s(M) = M\otimes_A\tau_s(A)$.
Hence (\ref{prop:LT_s RC_s and Ltau_s:distinguished triangle of LT_s and Ltau_s}) follows.

(\ref{prop:LT_s RC_s and Ltau_s:finiteness of RC_s}) follows from (\ref{prop:LT_s RC_s and Ltau_s:distinguished triangle of RC_s and Ltau_s}) and Theorem~\ref{thm:duality of Zuckerman functor} (\ref{thm:duality of Zuckerman functor:Vanishing of higher cohomology}).
Since $C_s$ is the right adjoint functor of $T_s$, we have $C_s(M) = \Hom(A,C_s(M)) = \Hom(T_s(A),M)$.
Hence we have $RC_s(M) = R\Hom(T_s(A),M)$.
By the exact sequence $0\to T_s(A)\to A\to \tau_s(A)\to 0$, we have a distinguished triangle $R\Hom(\tau_s(A),M)\to M\to RC_s(M)\xrightarrow{+1}$.
We have $R\Hom(\tau_s(A),M) = R\Hom(L\tau_s(A),M) = R\Hom(A,L\tau_s(M)[-2]) = L\tau_s(M)[-2]$ by Theorem~\ref{thm:duality of Zuckerman functor}.
Hence (\ref{prop:LT_s RC_s and Ltau_s:distinguished triangle of RC_s and Ltau_s}) follows.
We prove (\ref{prop:LT_s RC_s and Ltau_s:cohomology of T_s and C_s}).
From (\ref{prop:LT_s RC_s and Ltau_s:distinguished triangle of LT_s and Ltau_s}) and (\ref{prop:LT_s RC_s and Ltau_s:distinguished triangle of RC_s and Ltau_s}), we have $L^1T_sM = L^2\tau_sM = \Ker(M\to C_sM) = \Ker(M \to \varphi_sM)$.
We also have $R^1C_sM = \tau_sM = \Cok(T_sM\to M) = \Cok(\varphi_sM\to M)$.
\end{proof}

\begin{cor}\label{cor:T_s is a twisting functor}
Assume that $(W,S)$ is the Weyl group of a semisimple Lie algebra $\mathfrak{g}$.
From a result of Soergel~\cite{MR1029692}, the regular integral block of the BGG category $\mathcal{O}^{\mathrm{BGG}}$ of $\mathfrak{g}$ is equivalent to the category of finitely generated $A$-modules (Remark~\ref{rem:relation between O and BGG category}).
We regard $\mathcal{O}^{\mathrm{BGG}}$ is a full-subcategory of $\mathcal{O}$.
Then $T_s$ coincides with the twisting functor~\cite{MR1474841} and $C_s$ coincides with the Joseph's Enright functor~\cite{MR664114} on $\mathcal{O}^{\mathrm{BGG}}$.
\end{cor}
\begin{proof}
Since $C_s$ is the right adjoint functor of $T_s$ (Theorem~\ref{thm:adjointness of T_s and C_s}) and the Joseph's Enright functor is the right adjoint functor of the twisting functor~\cite[Theorem~3]{MR2115448}, the statement for $C_s$ follows from that for $T_s$.

From Proposition~\ref{prop:LT_s RC_s and Ltau_s} (\ref{prop:LT_s RC_s and Ltau_s:distinguished triangle of LT_s and Ltau_s}), for a projective object $P$, we have the following exact sequence:
\[
	0\to T_sP \to P\to \tau_sP\to 0.
\]
The twisting functor $T'_s$ satisfies the same exact sequence~\cite[Proposition~2.4 (1)]{MR2331754}.
Hence $T_sP\simeq T'_sP$.
Taking a projective resolution, we have $T_sM\simeq T'_sM$ for $M\in\mathcal{O}'$.
\end{proof}

\begin{prop}\label{prop:natural transformation of T_s, Verma modules case}
Assume that $sx > x$.
Then we have $T_sM(x) = M(sx)$ and $L^1T_sM(x) = 0$.
Moreover, a natural transformation $M(sx)\to M(x)$ is injective.
\end{prop}
\begin{proof}
This proposition follows from Lemma~\ref{lem:calculation of tau_s(M(e))} and Proposition~\ref{prop:LT_s RC_s and Ltau_s} (\ref{prop:LT_s RC_s and Ltau_s:cohomology of T_s and C_s}).
\end{proof}

\begin{prop}\label{prop:C_sM(x)}
We have
\[
	C_sM(x) = 
	\begin{cases}
	M(sx) & (sx < x),\\
	M(x) & (sx > x).
	\end{cases}
\]
\end{prop}
\begin{proof}
This proposition follows from Lemma~\ref{lem:calculation of tau_s(M(e))}.
\end{proof}

%

\section{Homomorphisms between Verma modules}\label{sec:Homomorphisms between Verma modules}
In this section, we prove the following theorem.
\begin{thm}\label{thm:Homomorphisms between Verma modules}
We have
\[
	\Hom(M(x),M(y)) = 
	\begin{cases}
	\C & (y \le x),\\
	0 & (y\not\le x).
	\end{cases}
\]
Moreover, any nonzero homomorphism $M(x)\to M(y)$ is injective.
\end{thm}
The surjective map $P(x)\to M(x)$ induces an injective map $\Hom(M(x),M(y))\to \Hom(P(x),M(y))$.
If $y\not\le x$, then
\begin{align*}
\Hom(P(x),M(y)) & = \Hom(\Phi(B(x)),\Phi(V(y)))\\
& = \Hom_Z(B(x),V(y))\otimes_{S(V^*)}\C\\
& = \Hom_{S(V^*)}(\mathscr{B}(x)_y,S(V^*))\otimes_{S(V^*)}\C = 0.
\end{align*}
Hence we get the theorem in the case of $y\not\le x$.

Next, we prove the `existence part' of the theorem.
Namely, we prove the following lemma.
\begin{lem}
If $y \le x$, then there exists an injective map $M(x)\to M(y)$.
\end{lem}
If $x = sy$, this lemma follows from Proposition~\ref{prop:natural transformation of T_s, Verma modules case}.
Hence, to prove the lemma, it is sufficient to prove the following lemma (see the proof of \cite[7.6.11.\ Lemma]{MR1393197}).
\begin{lem}
Let $s$ be a simple reflection and $x,y\in W$.
Assume that there exists an injective map $f\colon M(x)\to M(y)$.
If $sx > x$ then there exists an injective map $M(sx)\to M(sy)$.
\end{lem}
\begin{proof}
By Proposition~\ref{prop:natural transformation of T_s, Verma modules case}, there exists an injective map $M(sx)\to M(x)$.
If $sy > y$, then there exists an injective map $M(y)\to M(sy)$.
Hence the lemma follows.

We may assume that $sy < y$.
By Proposition~\ref{prop:natural transformation of T_s, Verma modules case}, we have $T_sM(x) = M(sx)$ and $T_sM(y) = M(sy)$.
Hence we get the following diagram:
\[
\xymatrix{
M(x)\ar[r]^f & M(y)\\
M(sx)\ar[u]\ar[r]^{T_sf} & M(sy).\ar[u]
}
\]
The vertical maps are the natural transformations and they are injective by Proposition~\ref{prop:natural transformation of T_s, Verma modules case}.
Hence $T_sf$ is injective.
\end{proof}
To prove Theorem~\ref{thm:Homomorphisms between Verma modules}, it is sufficient to prove the following lemma.
\begin{lem}
We have $\dim\Hom(M(x),M(y))\le 1$.
\end{lem}
\begin{proof}
We prove by induction on $\ell(x)$.
If $x = e$, then $M(x) = M(e) = P(e) = \Phi(B(e))$.
Hence we have 
\begin{multline*}
\Hom(M(e),M(y)) = \Hom(\Phi(B(e)),\Phi(V(y)))\\
= \Hom_Z(B(e),V(y))\otimes_{S(V^*)}\C = \Hom_{S(V^*)}(\mathscr{B}(e)_y,V(y))\otimes_{S(V^*)}\C.
\end{multline*}
If $y \ne e$, then this space is zero.
If $y = e$, then this space is $\C$.

Assume that $x\ne e$.
Take a simple reflection $s$ such that $sx < x$.
Then we have $M(x) = T_sM(sx)$ (Proposition~\ref{prop:natural transformation of T_s, Verma modules case}).
Since $C_s$ is the right adjoint functor of $T_s$, we have
\[
	\Hom(M(x),M(y)) = \Hom(T_sM(sx),M(y)) = \Hom(M(sx),C_sM(y)).
\]
If $sy > y$, then $C_sM(y) = M(sy)$.
If $sy < y$, then $C_sM(y) = M(y)$ (Proposition~\ref{prop:C_sM(x)}).
In each case, the dimension of this space is less than or equal to $1$ by inductive hypothesis.
\end{proof}

\section{More about the functors $T_s$ and $C_s$}\label{sec:More about the functors T_s C_s}
\begin{lem}
Let $s$ be a simple reflection and $x\in W$.
\begin{enumerate}
\item We have $L^1T_sM(x) = 0$.
\item The natural transformation $M(x)\to RC_sLT_sM(x)$ is an isomorphism.
\end{enumerate}
\end{lem}
\begin{proof}
By Proposition~\ref{prop:LT_s RC_s and Ltau_s} (\ref{prop:LT_s RC_s and Ltau_s:cohomology of T_s and C_s}), we have $L^1T_sM(x) = \Ker(M(x)\to \varphi_sM(x))$.
By Lemma~\ref{lem:calculation of tau_s(M(e))}, the last module is zero.

To prove (2), first we prove that $RC_sT_sM(x) \simeq M(x)$.
If $sx > x$, then $T_sM(x) = M(sx)$.
Hence $C_sT_sM(x) = C_sM(sx) = M(x)$ by Proposition~\ref{prop:C_sM(x)}.
By Proposition~\ref{prop:LT_s RC_s and Ltau_s} (\ref{prop:LT_s RC_s and Ltau_s:cohomology of T_s and C_s}) and Proposition~\ref{prop:translation of M(e), etc...}, we have $R^1C_sM(x) = \Cok(\varphi_sM(sx)\to M(sx)) = 0$.

Next, assume that $sx < x$.
First we prove that $R^1C_sT_sM(x) = 0$.
By Proposition~\ref{prop:LT_s RC_s and Ltau_s} (\ref{prop:LT_s RC_s and Ltau_s:distinguished triangle of RC_s and Ltau_s}), we have $R^1C_sT_sM(x) = \tau_sT_sM(x)$.
To prove $\tau_sT_sM(x) = 0$, it is sufficient to prove that $\Hom(T_sM(x),M) = 0$ for all $M\in\mathcal{O}_s$.
Since $C_s$ is the right adjoint functor of $T_s$, we have $\Hom(T_sM(x),M) = \Hom(M(x),C_sM)$.
By Lemma~\ref{lem:characterization of O_s and translation preserve O_s}, we have $\varphi_sM = 0$.
This implies $C_sM = 0$.
Hence $\Hom(T_sM(x),M) = 0$.

Using the natural transformation $M(x)\simeq T_sM(sx)\to M(sx)$, we regard $M(x)$ as a submodule of $M(sx)$.
By the definition of $T_s$ and Lemma~\ref{lem:calculation of tau_s(M(e))}, we have an exact sequence
\[
	0\to M(sx)/M(x)\to T_sM(x)\to M(x)\to 0.
\]
Since $M(sx)/M(x) \in \mathcal{O}_s$ (Lemma~\ref{lem:calculation of tau_s(M(e))}), $\varphi_s(M(sx)/M(x)) = 0$.
From the definition of $C_s$ and Proposition~\ref{prop:LT_s RC_s and Ltau_s} (\ref{prop:LT_s RC_s and Ltau_s:cohomology of T_s and C_s}), $C_s(M(sx)/M(x)) = 0$ and $R^1C_s(M(sx)/M(x)) = M(sx)/M(x)$.
Hence from the long exact sequence, we have
\[
	0  \to C_sT_sM(x)\to C_sM(x)\to M(sx)/M(x) \to  0.
\]
From Proposition~\ref{prop:C_sM(x)}, we have $C_sM(x) = M(sx)$.
Hence $C_sT_sM(x) \simeq M(x)$.

Since $\End(M(x)) = \C\id$ by Theorem~\ref{thm:Homomorphisms between Verma modules}, the natural transformation $M(x)\to RC_sLT_sM(x)$ is zero or an isomorphism.
Since this natural transformation comes from $\id\colon T_sM(x)\to T_sM(x)$ and the adjointness, this is not zero.
\end{proof}

\begin{thm}\label{thm:LT_s is an auto-equivalence}
The functor $LT_s$ gives an auto-equivalence of $D(\mathcal{O})$.
Its quasi-inverse functor is $RC_s$.
\end{thm}
\begin{proof}
We prove that the natural transformation $M\to RC_sLT_sM$ is an isomorphism for $M\in D(\mathcal{O})$.
Taking a projective resolution, we may assume that $M$ is a projective module.
Since a projective module has a filtration whose successive quotients are Verma modules, we may assume that $M$ is a Verma module.
This is proved in the previous lemma.
\end{proof}

\begin{thm}\label{thm:T_s sonnano-kankei-nee}
Let $w = s_1\dotsm s_l$ be a reduced expression of $w\in W$.
Then $T_{s_1}\dotsm T_{s_l}$ and $C_{s_1}\dotsm C_{s_l}$ is independent of the choice of a reduced expression.
\end{thm}
\begin{proof}
The statement for $C_s$ follows from the statement for $T_s$ (Theorem~\ref{thm:adjointness of T_s and C_s}).

Put $F = T_{s_1}\dotsm T_{s_l}$.
Take an another reduced expression $w = s_1'\dotsm s_l'$ and put $G = T_{s_1'}\dotsm T_{s_l'}$.
We use (the dual of) the comparison lemma~\cite[Lemma~1]{MR2115448}.
Namely, for a projective module $P$, we prove the following statements.
\begin{enumerate}
\item The natural transformations $FP\to P$ and $GP\to P$ are injective.
\item $FP\simeq GP$.
\item $\im(FP\to P) = \im(GP\to P)$.
\end{enumerate}
We may assume $P = P(x)$ for some $x\in W$.
We prove by induction on $\ell(x)$.

If $x = e$, then $P(x) = M(e)$.
By Proposition~\ref{prop:natural transformation of T_s, Verma modules case}, we have $FM(e) = GM(e) = M(w)$.
Hence we get (2).
We prove (1) by induction on $l$.
Put $F' = T_{s_2}\dotsm T_{s_l}$.
The natural transformation $FP\to P$ is given by $FP = T_{s_1}F'P\to F'P\to P$.
The natural transformation $F'P\to P$ is injective by inductive hypothesis.
Since $F'P = M(s_2\dotsm s_l)$, $T_{s_1}F'P\to F'P$ is injective (Proposition~\ref{prop:natural transformation of T_s, Verma modules case}).
Hence $FP\to P$ is injective.
Since $\dim\Hom(FM(e),M(e)) = \dim\Hom(M(w),M(e)) = 1$ by Theorem~\ref{thm:Homomorphisms between Verma modules}, we get (3).

Assume that $x\ne e$ and take a simple reflection $t$ such that $xt < x$.
Then $P = P(xt)$ satisfies (1--3).
By Theorem~\ref{thm:theta and varphi commute, in O}, $T_s$ commutes with $\theta_t$.
Hence $P = \theta_tP(xt)$ satisfies (1--3).
Since $P(x)$ is a direct summand of $\theta_tP(xt)$, $P = P(x)$ satisfies (1--3).
\end{proof}

\def\cprime{$'$} \def\dbar{\leavevmode\hbox to 0pt{\hskip.2ex
  \accent"16\hss}d}\newcommand{\noop}[1]{}


\begin{thebibliography}{BGG76}

\bibitem[Ark97]{MR1474841}
S.~M. Arkhipov, \emph{Semi-infinite cohomology of associative algebras and bar
  duality}, Internat. Math. Res. Notices (1997), no.~17, 833--863.

\bibitem[AS03]{MR2032059}
Henning~Haahr Andersen and Catharina Stroppel, \emph{Twisting functors on
  {$\mathcal{O}$}}, Represent. Theory \textbf{7} (2003), 681--699 (electronic).

\bibitem[BGG76]{MR0407097}
I.~N. Bern{\v{s}}te{\u\i}n, I.~M. Gel{\cprime}fand, and S.~I. Gel{\cprime}fand,
  \emph{A certain category of {${\mathfrak g}$}-modules}, Funkcional. Anal. i
  Prilo\v zen. \textbf{10} (1976), no.~2, 1--8.

\bibitem[BM01]{MR1871967}
Tom Braden and Robert MacPherson, \emph{From moment graphs to intersection
  cohomology}, Math. Ann. \textbf{321} (2001), no.~3, 533--551.

\bibitem[Dix96]{MR1393197}
Jacques Dixmier, \emph{Enveloping algebras}, Graduate Studies in Mathematics,
  vol.~11, American Mathematical Society, Providence, RI, 1996, Revised reprint
  of the 1977 translation.

\bibitem[EW80]{MR563362}
T.~J. Enright and N.~R. Wallach, \emph{Notes on homological algebra and
  representations of {L}ie algebras}, Duke Math. J. \textbf{47} (1980), no.~1,
  1--15.

\bibitem[Fie08a]{MR2395170}
Peter Fiebig, \emph{The combinatorics of {C}oxeter categories}, Trans. Amer.
  Math. Soc. \textbf{360} (2008), no.~8, 4211--4233.

\bibitem[Fie08b]{MR2370278}
Peter Fiebig, \emph{Sheaves on moment graphs and a localization of {V}erma
  flags}, Adv. Math. \textbf{217} (2008), no.~2, 683--712.

\bibitem[Jos82]{MR664114}
A.~Joseph, \emph{The {E}nright functor on the {B}ernstein-{G}el\cprime
  fand-{G}el\cprime fand category {${\mathcal O}$}}, Invent. Math. \textbf{67}
  (1982), no.~3, 423--445.

\bibitem[KM05]{MR2115448}
Oleksandr Khomenko and Volodymyr Mazorchuk, \emph{On {A}rkhipov's and
  {E}nright's functors}, Math. Z. \textbf{249} (2005), no.~2, 357--386.

\bibitem[MS07]{MR2331754}
Volodymyr Mazorchuk and Catharina Stroppel, \emph{On functors associated to a
  simple root}, J. Algebra \textbf{314} (2007), no.~1, 97--128.

\bibitem[Soe90]{MR1029692}
Wolfgang Soergel, \emph{Kategorie {$\mathscr O$}, perverse {G}arben und
  {M}oduln \"uber den {K}oinvarianten zur {W}eylgruppe}, J. Amer. Math. Soc.
  \textbf{3} (1990), no.~2, 421--445.

\bibitem[Soe07]{MR2329762}
Wolfgang Soergel, \emph{Kazhdan-{L}usztig-{P}olynome und unzerlegbare
  {B}imoduln \"uber {P}olynomringen}, J. Inst. Math. Jussieu \textbf{6} (2007),
  no.~3, 501--525.

\bibitem[Ver68]{MR0218417}
Daya-Nand Verma, \emph{Structure of certain induced representations of complex
  semisimple {L}ie algebras}, Bull. Amer. Math. Soc. \textbf{74} (1968),
  160--166.

\end{thebibliography}

\end{document}